\documentclass[11pt]{article}

\usepackage[english]{babel} 
\usepackage[utf8]{inputenc} 
\usepackage[T1]{fontenc}
\usepackage{amsmath,amssymb,fullpage}
\usepackage[all,dvips]{xy}

\usepackage{tikz}
\usepackage{tikz,fullpage}
\usetikzlibrary{positioning, arrows.meta}
\usetikzlibrary{matrix}

\usepackage{enumerate}
\usepackage{mathtools}

\usepackage{graphicx}
\usepackage{color}
\definecolor{vert}{rgb}{0.02,0.4,0.10}
\usepackage{hyperref}
\hypersetup{
	colorlinks=true,                
	breaklinks=true,                
	linkcolor= vert,                
	citecolor= blue                
}

\usepackage{graphicx}
\graphicspath{ {images/} }
\usepackage{pgf,tikz,pgfplots}
\pgfplotsset{compat=1.14}
\usepackage{mathrsfs}
\usetikzlibrary{arrows}

\usepackage{makeidx}
\usepackage{pstricks}

\providecommand{\otherindexspace}[1]{}

\usepackage{authblk}
\usepackage{amsthm}

\newtheorem{theorem}{Theorem}[section]

\newtheorem{lemma}[theorem]{Lemma}
\newtheorem{proposition}[theorem]{Proposition}
\newtheorem{remark}[theorem]{Remark}

\newtheorem{corollary}[theorem]{Corollary}

\numberwithin{equation}{section}

\def\vp{\varepsilon}

\def \R{\mathbb {R}}
\def \N{\mathbb{N}}
\def \Z{\mathbb{Z}}

\def \C{\mathbb{C}}

\def \T{\mathcal{T}}
\def \cL{\mathcal{L}}

\def\E{\mathbb{E}}
\def\P{\mathbb{P}}
\def\lb{[\![}
\def\rb{]\!]}

\usepackage{fancyhdr}

\renewcommand{\Re}{\mathrm{Re}}

\usepackage{graphics}
\usepackage{graphicx}

\DeclareGraphicsExtensions{.eps,.bmp,.jpg,.pdf,.mps,.png,.gif}

\makeatletter
\def\titre{\@title}
\makeatother

\title{Depth profile of depth-weighted trees with bounded weights}

\author{Emmanuel Kammerer \thanks{Emmanuel College \& DPMMS, University of Cambridge, ek672@cam.ac.uk}
}
\date{\today}
\begin{document}

\maketitle

\begin{abstract}
 We study the depth-weighted random recursive trees introduced by Leckey, Mitsche and Wormald in the case where the weights are bounded from above and from below. We establish the scaling limit of the depth profile of these trees when the weights satisfy a law of large numbers. In particular, we obtain the scaling limit of the depth of these trees, generalising some results of Lichev, Linker, Lodewijks and Mitsche. We also answer negatively a question left open by the same authors, showing that the scaling limit of the depth does not hold in general. Our main tools are appropriately defined martingales combined with the general Edgeworth expansion for the profiles introduced by Kabluchko, Marynych and Sulzbach.
\end{abstract}


\section{Introduction}
Random trees have been used to model various real-world networks. One of the simplest models, the uniform attachment tree, has been extensively studied since its introduction in \cite{NR70}. See e.g.\@ \cite{Pit94} and the survey \cite{SM94} for a detailed account of its properties. Uniform attachment trees, also called random recursive trees, are built recursively and every new vertex is attached to an existing vertex chosen uniformly at random. Since then, numerous generalisations of this model have been introduced, one of them being the Hoppe trees \cite{LN13}, where the root has a weight $\theta>0$ while all the other vertices have a weight $1$ and every new vertex is attached to an existing vertex proportionally to its weight. This slight generalisation of uniform attachment trees is actually a particular case of the very general model of weighted recursive trees introduced previously in \cite{BV06}.

More recently, \cite{LMW20} introduced a different generalisation where each generation of the tree has a different weight. Such trees, called depth-weighted trees, can be useful to model situations where a new vertex tends to attach to vertices of some generations but not other ones. Let us give a precise definition of the model. Let $\N = \{0, 1, \ldots\}$ be the set of natural numbers. If $\T$ is a rooted tree and $v \in \T$ is a vertex, then the graph distance from the root to $v$ is called the depth (also often called the height) of $v$ in $\T$ and we denote it by $d(v)$. The depth of $\T$ is defined as the maximal depth of a vertex in $\T$ and will be denoted by $d(\T)$. Let $f: \N \to (0, \infty)$. To each vertex $v$ of a tree, we assign a weight $f(d(v))$ which only depends on its depth $d(v)$. The sequence of depth-weighted random recursive trees $(\T_n)$ is defined inductively as follows. We start with a tree $\T_0$ made of only one vertex $v_0$. Then, assume that $\T_n$ is constructed for some $n\in \N$. Conditionally on $\T_n$, let $V_n$ be a vertex of $\T_n$ chosen with probability proportional to its weight, i.e.
\[
\forall v \in \T_n, \qquad \qquad \P(V_n=v\vert \T_n) = \frac{f(d(v))}{\sum_{u \in \T_n} f(d(u))}.
\]
The tree $\T_{n+1}$ is obtained by attaching a new vertex $v_{n+1}$ to $V_n$.

Different choices of the weight function $f$ yield very different asymptotic geometries for $\T_n$. In Theorem 2.3 of \cite{LMW20}, it is for example shown that the expected depth $\E(d(\T_n))$ is a $\Theta(\log n)$ when the $f$ is bounded from above and from below while it is a $\Theta(n)$ when $f$ increases fast enough, with some intermediate behaviours. In \cite{LLLM26}, the asymptotic behaviour of the depth itself is studied and Theorems 1.4 and 1.6 of \cite{LLLM26} show that for convergent, periodic or slowly growing $f$, the depth $d(\T_n)$ is a.s.\@ equivalent to $e\log n$, mirroring the classical result for random recursive trees \cite{Pit94}. Their other main result, presented in Theorems 1.8 and 1.13 of \cite{LLLM26}, shows that if $f$ has an exponential or a super-exponential growth, then the depth of $\T_n$ is a.s.\@ a $\Theta(n)$ as $n\to \infty$, which is equivalent to $n$ in the case of super-exponential growth (and if $f$ has a sub-exponential growth, then the depth is a $o(n)$ almost surely). After this systematic analysis and a few simulations, they leave three appealing open questions. In the first one, Question 1.5, they expect the $e \log n$ asymptotics to hold more generally and ask whether it is always the case when $f$ is bounded from above and from below.

\paragraph{Main results.} In this paper, we work in the setting where $f$ is bounded from above and from below, i.e.\@ that there exists a constant $c>0$ such that $c \le f(n) \le 1/c$ for all $n \in \N$. We study the scaling limit of the depth as well as that of the profile of the tree.

In order to state our main result, let us introduce several quantities. We set $S_n \coloneqq \sum_{k=0}^{n-1} \log f(k)$ for all $n\ge 0$, where by convention we set $S_0=0$. We introduce the following condition:
\begin{equation}\label{condition ergodicite}
	\frac{S_n}{n} \mathop{\longrightarrow}_{n\to \infty} \ell \in \R.
\end{equation}
Note that the cases where $f$ is periodic or convergent explored in \cite{LLLM26} are particular cases of the above assumption. Another natural particular case of the above assumption is the case where the weights are taken i.i.d.\@ at random. \textbf{Note that since the law of $\T_n$ only depends on $f$ up to a multiplicative constant, we may assume that $\ell=0$ in \eqref{condition ergodicite}.}

We also denote by $Z_n \coloneqq \sum_{v \in \T_n} f(d(v))$ the sum of the weights of the vertices of $\T_n$. Note that $Z_n$ is random.  Let $n_0\coloneqq \lfloor e^2/c \rfloor +1$. A quantity of interest will be the sum of the inverses of $Z_n$ defined for all $n\ge n_0$ by
\begin{equation}\label{eq h n}
	\mathrm{h}_n \coloneqq \sum_{k=n_0}^{n-1} \frac{1}{Z_k}.
\end{equation}
Note that since $c \le f \le 1/c$, we have
\begin{equation}\label{eq encadrement h n}
	c\sum_{k=n_0}^{n-1} \frac{1}{k+1} \le \mathrm{h}_n \le \frac{1}{c}\sum_{k=n_0}^{n-1} \frac{1}{k+1}.
\end{equation}

Our main result can be stated as follows. For all $n ,k \in \N$, let $\mathbb{L}_n(k)$ be the number of vertices of $\T_n$ at depth $k$. Recall that a point $z$ of a set $\mathcal{Z}\subset \R$ is isolated if there exists $\vp>0$ such that $(z-\vp, z+\vp) \cap \mathcal{Z}= \{z\}$.
\begin{theorem}\label{theoreme profil}
	Assume that $f$ is bounded from above and from below and satisfies \eqref{condition ergodicite} with 
	$\ell=0$. Then, there exist a random variable $W_\infty(0)>0$ and a random set $\mathcal{Z}\subset (-\infty, 1)\setminus \{0\}$ which is at most countable and whose elements are isolated such that the following holds. Almost surely, as $n\to \infty$,
	\[\mathbb{L}_n(k) = \frac{W_\infty(0)e^{\mathrm{h}_n}}{\sqrt{2\pi \mathrm{h}_n}} \exp\left(S_k- \frac{1}{2} \left(\frac{k-\mathrm{h}_n}{\sqrt{\mathrm{h}_n}}\right)^2 \right)  + e^{S_k}o\left(\frac{e^{\mathrm{h}_n}}{\sqrt{\mathrm{h}_n}} \right),\]
	where the small $o$ is uniform in $k \in \N$. Moreover, for every compact set $K \subset (-\infty,1)$, on the event that $ K \cap \mathcal{Z} = \emptyset $, we have almost surely for all $z \in K$,
	\[
	\mathbb{L}_n(\lfloor e^z \mathrm{h}_n \rfloor) = \exp\left({S_{\lfloor e^z \mathrm{h}_n \rfloor} +  (1-z)e^z \mathrm{h}_n -\frac{1}{2} \log \mathrm{h}_n + O(1) }\right),
	\]
	where the $O(1)$ is uniform in $z\in K$. Furthermore,
	\[
	\frac{\mathrm{h}_n}{\log n} \mathop{\longrightarrow}\limits_{n\to \infty}^{\mathrm{a.s.}} 1 \qquad \text{and} \qquad \frac{d(\T_n)}{\log n} \mathop{\longrightarrow}\limits_{n\to \infty}^{\mathrm{a.s.}} e.
	\]
\end{theorem}
Actually, the set $\mathcal{Z}$ is obtained as the set of zeros on $(-\infty, 1)$ of a random analytic function $z\mapsto W_\infty(z)$ such that $W_\infty(0) > 0$ almost surely. A notable feature of this result is the appearance of the term $S_k$ in the scaling limit of the profile, which was not present in the case of uniform random recursive trees or weighted random recursive trees (see Theorem 5 of \cite{Sen21}). For instance, if the weights are i.i.d.\@, then the central limit theorem yields that the fluctuations of $S_{\lfloor \log n \rfloor}$ are of order $\sqrt{\log n}$, in contrast to the classical $\log \log n$ corrections observed in branching random walks and random recursive trees. Another difference with usual random recursive trees is that the profile may vary significantly from one depth to the next one. Indeed, the first point of the above theorem shows that a.s.\@ $\mathbb{L}_n(k+1)/\mathbb{L}_n(k) \to f(k)$ uniformly in $k$ such that $\vert k - \mathrm{h}_n\vert \le C \sqrt{\log n}$ for any fixed $C>0$. Note that the last point of the above theorem generalises Theorem 1.4 of \cite{LLLM26}.

\begin{remark}The precise behaviour of $\mathrm{h}_n$ heavily depends on the choice of $f$, but we show in Lemma \ref{lemme equivalent ps h n} that a.s.\@ $\mathrm{h}_n = \log n+ O(\max_{0 \le k \le d(\T_n)} \vert S_k\vert)$ as $n\to \infty$. For example, when $f$ is periodic, we get $\mathrm{h}_n= \log n+O(1)$ a.s.\@ and when the weights are i.i.d.\@ we have $\mathrm{h}_n = \log n + O(\sqrt{\log (n) })$ in probability (see Remark \ref{more precise behaviour of h n} for the more precise behaviour of $\mathrm{h}_n$ in this latter case).
\end{remark}
The proof of Theorem \ref{theoreme profil} also gives the depth of a typical vertex, whose first-order behaviour is the same as for uniform random recursive trees. For all $n\ge 0$, conditionally on $\T_n$, let $V_n$ be a uniform random vertex of $\T_n$.
\begin{theorem}\label{th hauteur sommet typique}
	Under the same assumptions as in Theorem \ref{theoreme profil},
	\[
	\frac{d(V_n)}{\log n} \mathop{\longrightarrow}\limits_{n\to \infty}^{(\P)} 1.
	\]
	More precisely, for all $\vp>0$, almost surely, the number of vertices at depth between $(1-\vp)\log n $ and $(1+\vp )\log n $ is equivalent to $n$.
\end{theorem}

One can wonder what happens when \eqref{condition ergodicite} is not satisfied. In general, for weights bounded from above and from below, the scaling limit of the depth does not hold. We provide a simple counterexample in the general case with a function $f$ which only takes two values.
\begin{proposition}\label{lemme contre-exemple}
	There exists a weight function $f$ taking its values in $\{1,3\}$ such that, almost surely,
	\[
	\liminf_{n\to \infty} \frac{d(\T_n)}{\log n} \le \frac{5}{2}
	\qquad \text{and} \qquad \limsup_{n\to \infty} \frac{d(\T_n)}{\log n} \ge e.
	\]
\end{proposition}

\paragraph{Another related general model of recursive trees.} As mentioned at the beginning of the introduction, a natural generalisation of random recursive trees is the model of weighted recursive trees, first introduced in \cite{BV06}, where the vertex $v_n$ added at time $n$ has a weight $w_n>0$. The study of this model is much more advanced than the one of depth weighted trees, see e.g.\@ \cite{MU19, HI20, Sen21, PS22, PS24}. See also \cite{BMMU24} for a result which can be reformulated in terms of weighted recursive trees, as well as \cite{ELO23, Iye24, Lod24} for some results in the case of i.i.d.\@ weights. In \cite{Sen21}, the scaling limit of the height and that of the profile are established under general conditions. It is also worth mentioning that \cite{Sen21} presents a connection to affine preferential attachment trees which enables him to transfer the results for weighted recursive trees to preferential attachment trees.

\paragraph{Techniques.} Our approach is inspired by \cite{Sen21}. We first introduce appropriately defined martingales to study a weighted Laplace transform of the profile. Using the general framework of \cite{KMS17}, we then establish an Edgeworth expansion of the profile, which is a very precise control of the asymptotic behaviour of the profile. We refer the reader to \cite{KMS17} for a broad presentation of this tool and for related works. Note that, despite its generality, the road opened by \cite{KMS17} has been used in very few papers (to our knowledge, the only paper using the general Edgeworth expansion of \cite{KMS17} is \cite{Sen21}, and closely related ideas appear in \cite{KKS25}). The purpose of our paper is also to promote the use of \cite{KMS17}. We state this expansion in Subsection \ref{sous-section preuve du profil}. It is more precise than our Theorem \ref{theoreme profil} but we did not choose to include it in the introduction since it is expressed in a less explicit way. We do not rely on the coupling with continuous time processes used by \cite{LMW20, LLLM26} and our paper is independent of their results. We follow roughly the same approach as in \cite{Sen21, KKS25}, first proving the convergence of the martingales to a random analytic function, then studying the zeros of the analytic function to finally get back to the profile of the tree. Contrary to \cite{Sen21}, we do not need to re-weight the profile since in depth weighted trees, the weights only depend on the depth. However we do need to overcome new difficulties. The first obstacle is the fact that the quantities $Z_n$ and $\mathrm{h}_n$ are random, they depend on the tree $\T_n$, and their asymptotic behaviour is not known at the beginning. To deal with this obstacle, we first establish rough bounds in expectation to control the martingale and then improve them to more precise almost sure asymptotics. Another complication arises in the study of the zeros of the limit, i.e.\@ the zeros of the function $W_\infty$ mentioned below Theorem \ref{theoreme profil}. The standard way to show that $W_\infty$ has no zero on the real axis is to exploit the self-similarity: showing that the limit satisfies some kind of fixed point equation; the idea originates from \cite{BG79} in the context of branching random walks and some analogous arguments for random trees can be found in Lemma 21 of \cite{Sen21} and in Subsection 4.7 of \cite{KKS25}. But in our framework, the limit of the martingale does not satisfy such a fixed point equation (except if one makes further assumptions such as assuming that $f$ is periodic). We therefore use a different method: by controlling the logarithm of our martingales, we prove that $W_\infty$ has almost surely at most countably many zeros which are isolated. We leave the question of proving or disproving that $W_\infty$ does not have any zeros as an open problem.
\paragraph{Further directions.} The general Edgeworth expansion only applies to trees with logarithmic depth. In this context, the (weighted) Laplace transform and the martingales that we introduce could be useful in a more general setting as soon as one understands the behaviour of the sum of the weights $\sum_{v \in \T_n} f(d(v))$. It would then seem feasible to deal with slowly growing (see Theorem 1.6 of \cite{LLLM26}) or polynomial (see Theorem 2.3 (b) of \cite{LMW20}) weights using the same techniques. Another question which was investigated in \cite{PS22} in the case of weighted recursive trees is the next order in the asymptotic behaviour of the depth. Besides, another question of interest is the asymptotic behaviour of the distance between two typical vertices. Under the conditions of Theorem \ref{theoreme profil}, we expect this distance to be equivalent in probability to $2 \log n$.
\paragraph{Outline.} In Section \ref{section martingale}, we introduce a weighted Laplace transform of the profile and its associated martingale. The next section, Section \ref{section profil}, is the main part of the paper and is devoted to the proof of Theorem \ref{theoreme profil}. We show the convergence of the martingales in Subsection \ref{sous-section controle martingale}. Then, in Subsection \ref{sous-section zeros de la limite}, we study the zeros of the limit. Next, Subsection \ref{sous-section concentration} proves that the martingale concentrates on the real axis. In Subsection \ref{sous-section preuve du profil}, we obtain the general Edgeworth expansion of the profile and deduce in particular Theorems \ref{theoreme profil} and \ref{th hauteur sommet typique}. Finally, Section \ref{section contre-exemple} provides the counterexample mentioned in Proposition \ref{lemme contre-exemple}. 
\section{Laplace transform and martingales}\label{section martingale}
Let us introduce the following weighted Laplace transform of the profile, defined for all $n\in \N$ and $z \in \C$ by
\[
\cL_n(z) \coloneqq \sum_{v \in \T_n} \frac{1}{\prod_{k=0}^{d(v)-1} f(k)} e^{zd(v)}  = \sum_{v \in \T_n} e^{zd(v)-S_{d(v)}} = \sum_{k=0}^{\infty} \mathbb{L}_n(k)e^{zk-S_k},
\]
where the sum runs over all the vertices of $\T_n$. Let $(\mathcal{F}_n)_{n\ge 0}$ be the natural filtration associated to $(\T_n)_{n\ge 0}$.
\begin{lemma}\label{lemme martingale}
	For all $z \in \C$, for all $n \in \N$,
	\[
	\E\left(\left. \cL_{n+1}(z)\right\vert \mathcal{F}_n \right) = \left(1 + \frac{e^z}{Z_n}\right) \cL_n(z).
	\]
\end{lemma}
\begin{proof}
	Using the fact that conditionally on $\T_n$, the vertex $v_{n+1}$ which is attached at time $n+1$ is attached to the vertex $v \in \T_n$ with probability proportional to $f(d(v))$, we get
	\begin{align*}
		\E\left(\left. \cL_{n+1}(z)\right\vert \mathcal{F}_n \right)&=
		\cL_n(z) + \sum_{v \in \T_n} \frac{f(d(v))}{Z_n} \frac{1}{\prod_{k=0}^{d(v)+1-1} f(k)} e^{z (d(v)+1)}
		\\
		&=
		\left(1+ \frac{e^z}{Z_n} \right) \cL_n(z),
	\end{align*}
	thus ending the proof.
\end{proof}
From now on, we work under the assumption that $c\le f\le 1/c$ for some constant $c>0$. Recall that $n_0= \lfloor e^2/c \rfloor +1$. Let us consider the two complex half-planes
\[
\mathcal{E} \coloneqq \{ z \in \C; \ \Re(z) <1\} \qquad \text{and} \qquad
\mathcal{E}' \coloneqq \{ z\in \C; \ \Re(z) < 2\}.
\]
Note that for all $z \in \mathcal{E}'$, for all $n\ge n_0$, we have $1+ e^z/Z_n \neq 0$ since the real part is strictly positive. For all $n\ge n_0$, for all $z\in \C$, let
\[
C_n(z) \coloneqq  \prod_{k=n_0}^{n-1} \left( 1 + \frac{e^z}{Z_k}\right).
\]
Then, by Lemma \ref{lemme martingale}, for all $z\in \mathcal{E}'$, the process $(M_n(z))_{n\ge n_0}$ defined by setting for all $n\ge n_0$,
\[
M_n(z) \coloneqq \frac{1}{C_n(z)} \cL_n(z),
\]
is a martingale with respect to $(\mathcal{F}_n)_{n\ge0}$.

\begin{lemma}\label{lemme equivalent C n}
	Fix a compact subset $K \subset \mathcal{E}'$. Then,
	\[
	C_n(z) =  \exp\left(e^z\mathrm{h}_n +O(1) \right) \text{ and more precisely }
	C_n(z) = \exp\left( e^z \mathrm{h}_n+ c(z) +o(1) \right)
	\]
	uniformly in $z \in K$ as $n\to \infty$, where the $O(1)$ is a sequence of random variables bounded in absolute value by a (deterministic) constant which only depends on $K$ and $f$, where $c(z)$ is a random analytic function on $\mathcal{E}'$ and where the $o(1)$ is a sequence of random variables whose absolute value is smaller than a deterministic sequence converging to zero. In particular, if $K \subset \mathcal{E}$, then for all $p \in (1,2)$ and $\vp \in [0, 2-p)$,
	\[
	\frac{C_n(p \Re(z) + \vp)}{\vert C_n(z)\vert^p} = \exp\left( \left(e^{p \Re(z)+ \vp} - p \Re(e^z)\right) \mathrm{h}_n + O(1)\right)
	\]
	uniformly in $z \in K$ as $n\to \infty$ (where the $O(1)$ is a sequence of random variables bounded in absolute value by a deterministic constant which only depends on $K, p, \vp$ and $f$).
\end{lemma}
\begin{proof}
	Let $n_1 \coloneqq \lfloor 2e^2/c \rfloor +1$, so that for all $n\ge n_1$, for all $z \in \mathcal{E}'$, 
	we have $
	\left\vert  {e^z}/({nc}) \right\vert <{1}/{2}$. 
	Note that for all $z \in \mathcal{E}'$, for all $n\ge n_1$, we have $\vert e^z/Z_n\vert <1/2$. For all $z\in \C$ such that $\vert z\vert <1$, we write $\log (1+z) = \sum_{n\ge 1} {(-1)^{n-1}}z^n/n$. Then, for $n\ge n_1$, uniformly in $z \in K$,
	\begin{align*}
		C_n(z)&=\prod_{k=n_0}^{n_1-1} \left( 1+ \frac{e^z}{Z_k} \right)  \times   \prod_{k=n_1}^{n-1}  \left( 1+ \frac{e^z}{Z_k} \right)  \\
		&=e^{O(1)} \times \exp\left( \sum_{k=n_1}^{n-1} \log \left(1+ \frac{e^z}{Z_k}  \right) \right)\\
		&= \exp\left( \sum_{k=n_1}^{n-1} \frac{e^z}{Z_k} +O(1) \right),
	\end{align*}
	where the $O(1)$ is a sequence of random variables bounded in absolute value by a constant which only depends on $f$ and $K$, where in the last line we used the fact that $\vert \log (1+x) - x\vert\le  \vert x\vert^2$ for all $\vert x \vert \le 1/2$ together with the fact that
	\begin{equation}\label{eq cv somme des un sur W}
	\sum_{k=n_1}^\infty \frac{\vert e^{2z}\vert}{Z_k^2}\le \left(\sup_{z \in \mathcal{E}'}\vert e^{2z}\vert\right) \sum_{k=n_1}^\infty \frac{1}{c^2k^2}.
	\end{equation}
	Using \eqref{eq h n}, we deduce that $C_n(z)=\exp(e^z \mathrm{h}_n +O(1))$ uniformly in $z \in K$, where the $O(1)$ is a sequence of random variables bounded in absolute value by a constant which only depends on $f$ and $K$.
	
	For the more precise asymptotic behaviour, we proceed as follows. Recall that for all $z \in \mathcal{E}'$, for all $n\ge n_0$, we have $1+e^z/Z_n\neq 0$. Therefore, there is an analytic function $z\mapsto c_1(z)$ on $\mathcal{E}'$ such that for all $z \in \mathcal{E}'$, 
	\[
	\prod_{k=n_0}^{n_1-1} \left( 1+ \frac{e^z}{Z_k} \right)  = e^{c_1(z)}.
	\]
	We then set for all $z \in \mathcal{E}'$,
	\[
	c(z) \coloneqq c_1(z) - \sum_{k=n_0}^{n_1 -1} \frac{e^z}{Z_k}+ \sum_{k=n_1}^\infty \left( \log \left(1+ \frac{e^z}{Z_k} \right) - \frac{e^z}{Z_k}\right),
	\]
	where the series converges uniformly on $\mathcal{E}'$ by the inequality $\vert \log (1+x) - x\vert\le  \vert x\vert^2$ for all $\vert x \vert \le 1/2$ and by \eqref{eq cv somme des un sur W}.
	
	One can then write
	\begin{align*}
	C_n(z) &= \exp\left(c_1(z) + \sum_{k=n_1}^{n-1}  \log \left(1+ \frac{e^z}{Z_k} \right) \right)\\
	&=\exp\left(c(z) + \sum_{k=n_0}^{n_1 -1} \frac{e^z}{Z_k} +  \sum_{k=n_1}^{n -1} \frac{e^z}{Z_k}- \sum_{k=n}^\infty \left( \log \left(1+ \frac{e^z}{Z_k} \right) - \frac{e^z}{Z_k}\right) \right)\\
	&=\exp\left( c(z) + e^z \mathrm{h}_n -\sum_{k=n}^\infty \left( \log \left(1+ \frac{e^z}{Z_k} \right) - \frac{e^z}{Z_k}\right) \right),
	\end{align*}
	and we recognize the remainder of the uniformly converging series in the definition of $c(z)$. This proves the first part of the lemma. The second part of the lemma is a simple consequence of the first part, using the fact that $\vert e^z\vert  = e^{\Re(z)}$ for all $z \in \C$.
\end{proof}

\section{Depth profile for weights satisfying \eqref{condition ergodicite}}\label{section profil}
In this section, we assume that condition \eqref{condition ergodicite} is in force and we aim at proving Theorem \ref{theoreme profil}. 
\subsection{Control of the martingales $M_n(z)$}\label{sous-section controle martingale}
In this subsection, we give some estimates for the martingales $(M_n(z))_{n\ge n_0}$ and prove their convergence for $z$ close enough to the real axis. The following lemma controls the moments of the martingales.
\begin{lemma}\label{lemme moment Lp martingale}
	For all $p \in (1,2)$, for all $\vp \in (0, 2-p)$ , there exists a constant $C_\vp$ depending only on $f$ and $\vp$ such that we have for all $n\ge n_0$ and $z \in \mathcal{E}$,
	\[
	\vert M_n(z)\vert^p\le C_\vp \frac{n^{p-1} \cL_n(p \Re(z)+(p-1)\vp)}{\vert C_n(z)\vert^p}.
	\]
	Moreover, if we fix a compact $K\subset \mathcal{E}$, then
	\[
	\E\left( \frac{\cL_n(p \Re(z)+(p-1)\vp)}{\vert C_n(z)\vert^p} \right) \le  \exp\left(\left( e^{p \Re(z)+(p-1)\vp} - p\Re(e^z)\right) \log(n)/c + O(1)\right),
	\]
	uniformly in $z \in K$ as $n\to \infty$ where the $O(1)$ is a sequence bounded from above by a constant which only depends on $K$, $p$, $\vp$ and $f$.
\end{lemma}
\begin{proof}
	Recall that $M_n(z)=\cL_n(z)/C_n(z)$. By Jensen's inequality, we see that
	\[
	\left\vert \frac{\cL_n(z)}{n+1} \right\vert^p = \left\vert \frac{1}{n+1} \sum_{v \in \T_n} e^{-S_{d(v)}} e^{zd(v)} \right\vert^p\le  \frac{1}{n+1} \sum_{v \in \T_n} e^{-pS_{d(v)}} e^{p \Re(z) d(v)}.
	\]
	Moreover, by \eqref{condition ergodicite}, there exists $c_\vp>0$ such that for all $n\ge n_0$ and $z \in \mathcal{E}$, we have
	\begin{align*}
	\sum_{v \in \T_n} e^{- p S_{d(v)} + p \Re(z) d(v)} &\le  c_\vp \sum_{v \in \T_n} e^{-S_{d(v)}+(p-1)\vp d(v) + p\Re(z)d(v) } =  c_\vp \sum_{v \in \T_n} e^{-S_{d(v)} + ((p-1) \vp + p\Re(z)) d(v)}.
	\end{align*}
	This proves the first part of the lemma. Let us show the second part. One can write for $z \in K$,
	\begin{align*}
		\frac{\cL_n(p\Re(z) + (p-1)\vp)}{\vert C_n(z)\vert^p} &= \frac{C_n(p\Re(z) + (p-1) \vp)}{\vert C_n(z)\vert^p} M_n(p\Re(z) + (p-1) \vp)\\
		&= \exp\left( \left(e^{p\Re(z) + (p-1)\vp} -p\Re(e^z) \right) \mathrm{h}_n + O(1)\right) M_n(p\Re(z) + (p-1) \vp),
	\end{align*}
	where we applied Lemma \ref{lemme equivalent C n}, where the $O(1)$ is a sequence of random variables which is bounded from above by a constant which only depends on $K, \vp, p$ and $f$. Using \eqref{eq encadrement h n}, we deduce that 
	\[\frac{\cL_n(p\Re(z) + (p-1)\vp)}{\vert C_n(z)\vert^p} \le  \exp\left( \left(e^{p\Re(z) + (p-1)\vp} -p\Re(e^z) \right) \log(n)/c + O(1)\right) M_n(p\Re(z) + (p-1) \vp),\]
	where the $O(1)$ is a sequence of random variables which is bounded from above by a constant which only depends on $K, \vp, p$ and $f$. We conclude by taking the expectation.
\end{proof}
Next, let us control the increments of the martingales.
\begin{proposition}\label{prop controle increments martingale}
	Let $K \subset \mathcal{E}$ be a compact set. For all $p \in (1,2)$, for all $\vp \in (0,2-p)$,
	\[
	\E\left( \left\vert M_{n+1}(z)- M_n(z)\right\vert^p\right)\le \frac{1}{n}
	\exp\left( \left(e^{p \Re(z) +(p-1)\vp}-p\Re(e^z)\right) \log (n)/c + O(1)\right),
	\]
	uniformly in $z \in K$ as $n\to \infty$, where the $O(1)$ is a sequence bounded by a constant which only depends on $K, p, \vp$ and $f$. In particular, 
	\[
	\max_{k \in \lb 0, n \rb} \E\left( \left\vert M_{n+k}(z)- M_n(z)\right\vert^p\right)\le 
	\exp\left( \left(e^{p \Re(z) +(p-1)\vp}-p\Re(e^z)\right) \log (n)/c + O(1)\right),
	\]
	uniformly in $z \in K$ as $n\to \infty$, where the $O(1)$ is a sequence bounded by a constant which only depends on $K, p, \vp$ and $f$. 
\end{proposition}
\begin{proof}
	We write
	\[
	M_{n+1}(z) -
	M_n(z) = \frac{\cL_{n+1}(z)}{C_{n+1}(z)} - \frac{\cL_n(z)}{C_n(z)} = \left( \frac{C_n(z)}{C_{n+1}(z)} -1\right)M_n(z) + \frac{1}{C_{n+1}(z)} e^{zd(v_{n+1}) - S_{d(v_{n+1})}}.
	\]
	Then,
	\begin{align}
		\E\left( \left. \left\vert M_{n+1}(z)-M_n(z) \right\vert^p  \right\vert \mathcal{F}_n\right)\le &2^p\left\vert \frac{C_n(z)}{C_{n+1}(z)} -1\right\vert^p \vert M_n(z)\vert^p \label{premiere ligne increment martingale}\\
		&+\frac{2^p}{\vert C_{n+1}(z)\vert^p} \E\left(\left.   e^{p \Re(z) d(v_{n+1})-pS_{d(v_{n+1})}} \right\vert \mathcal{F}_n \right).\label{deuxieme ligne increment martingale}
	\end{align}
	For $z \in K$, the first line \eqref{premiere ligne increment martingale} can be rewritten
	\[
	2^p \frac{e^{p\Re(z)}}{Z_n^p} \frac{\vert \cL_n(z) \vert^p}{\vert C_{n+1}(z) \vert^p}  \le c(K,p) \frac{1}{n^p} \vert M_n(z)\vert^p,
	\]
	where $c(K,p)$ is a (deterministic) constant which only depends on $K,p$ and $f$. Moreover, for all $z \in K$, the second line \eqref{deuxieme ligne increment martingale} can be written
	\begin{align*}
		\frac{2^p}{\vert C_{n+1}(z)\vert^p} \sum_{v \in \T_n} \frac{f(d(v))}{Z_n} e^{p\Re(z)(d(v)+1)-p S_{d(v)+1}}&=\frac{2^p e^{p\Re(z)}}{Z_n\vert C_{n+1}(z)\vert^p} \sum_{v \in \T_n} e^{-S_{d(v)}} e^{p\Re(z)d(v) -(p-1) S_{d(v)+1}}\\
		&\le \frac{c(\vp, K, p)}{n\vert C_{n}(z)\vert^p} \cL_n(p\Re(z) + (p-1)\vp),
	\end{align*}
	for some constant $c(\vp, K, p)$ also depending on $f$.
	
	Thus, for all $n \ge n_0$, for all $z \in K$,
	\begin{equation}\label{eq majoration moment conditionnel increment}
	\E\left( \left. \left\vert M_{n+1}(z)-M_n(z) \right\vert^p  \right\vert \mathcal{F}_n\right)\le c(K,p) \frac{1}{n^p} \vert M_n(z) \vert^p + \frac{c(\vp, K, p)}{n} \frac{\cL_n(p\Re(z)+(p-1)\vp)}{\vert C_{n}(z)\vert^p}.
	\end{equation}
	By taking the expectation, we deduce that for all $n \ge n_0$, for all $z \in K$,
	\[
	\E\left( \vert M_{n+1}(z) -M_n(z)\vert^p \right)\le c(K,p)\frac{1}{n^p}\E\left( \vert M_n(z)\vert^p \right) + c(\vp, K,p)\frac{1}{n} \E\left(\frac{\cL_n(p\Re(z)+(p-1)\vp)}{\vert C_{n}(z)\vert^p}\right).
	\]
	By taking Lemma \ref{lemme moment Lp martingale} into account, we deduce that
	\[
	\E\left(\vert M_{n+1}(z)-M_n(z)\vert^p\right) \le\frac{1}{n} \exp\left(\left(e^{p\Re(z)+(p-1)\vp} -p\Re(e^z)\right) \log(n)/c +O(1)\right),
	\]
	uniformly in $z \in K$ as $n\to \infty$, where the $O(1)$ is a sequence bounded from above by a constant which only depends on $K,p,\vp$ and $f$. This is the first point of the proposition. The second point is obtained by taking the sum and using the inequality
	\begin{equation}\label{eq lemme Biggins}
	\forall z \in \mathcal{E}', n\ge n_0,  j \in \N, \qquad\E\left(\left.\left\vert M_{n+j}(z)- M_n(z) \right\vert^p \right\vert \mathcal{F}_n \right)\le 2^p \sum_{k=0}^{j-1}  \E\left(\left. \left\vert M_{n+k+1}(z)- M_{n+k}(z) \right\vert^p \right \vert \mathcal{F}_n\right),
	\end{equation}
	which is obtained by applying Lemma 1 of \cite{Big92} (see also Lemma 36 of \cite{Sen21}) to the martingale $(M_{n+k}(z)-M_n(z))_{k\ge 0}$ under the conditional probability measure $\P(\cdot\vert \mathcal{F}_n)$.
\end{proof}
We then record a useful lemma from \cite{Sen21}.
\begin{lemma}[Lemma 37 of \cite{Sen21}]\label{lemme 37}
	Suppose that $(z \mapsto \mathcal{M}_n(z))_{n\ge 1}$ is a sequence of analytic functions on some open set $\mathcal{O} \subset \C$, adapted to a filtration $(\mathcal{G}_n)_{n\ge 1}$. Suppose that for every $z \in \mathcal{O}$, the sequence $(\mathcal{M}_n(z))_{n\ge 1}$ is a martingale with respect to $(\mathcal{G}_n)_{n\ge 1}$. If there exists $q>1$ and continuous functions $\alpha:\mathcal{O} \to \R$ and $\delta: \mathcal{O} \to (0, \infty)$ such that for every compact subset $K \subset \mathcal{O}$, as $n\to \infty$,
	\[
	\sup_{z \in K}\E\left( \vert \mathcal{M}_{2n}(z)-\mathcal{M}_n(z) \vert^q \right) = O \left( n^{\alpha(z) q- \delta(z) }\right),
	\]
	then, for all compact subsets $K \subset \mathcal{O}$, there exists $\vp>0$ such that
	\begin{enumerate}[(i)]
		\item \label{item i lemme 37} If $\alpha>0$ on $\mathcal{O}$, then 
		\[
		\sup_{z \in K} n^{-\alpha(z)} \vert \mathcal{M}_n(z) - \mathcal{M}_1(z) \vert = O(n^{-\vp}) \text{ a.s.}\quad \text{and} \qquad  \sup_{z \in K} n^{-\alpha(z)}  \E\left(\vert \mathcal{M}_n(z) - \mathcal{M}_1(z) \vert  \right) = O(n^{-\vp});
		\]
		\item \label{item ii lemma 37} If $\alpha\le 0$ on $\mathcal{O}$, then $\mathcal{M}_n(z)$ converges a.s.\@ to a limit $\mathcal{M}_\infty(z)$ for all $z \in \mathcal{O}$ and we have 
		\[
		\sup_{z \in K} n^{-\alpha(z)} \vert \mathcal{M}_n(z) - \mathcal{M}_\infty(z) \vert = O(n^{-\vp}) \text{ a.s.}\quad \text{and} \qquad  \sup_{z \in K} n^{-\alpha(z)}  \E\left(\vert \mathcal{M}_n(z) - \mathcal{M}_\infty(z) \vert  \right) = O(n^{-\vp}).
		\]
	\end{enumerate}
\end{lemma}

Now equipped with the above, we can obtain the convergence of the martingales on some suitable domains. Let us write $\mathcal{E}_\beta \coloneqq\{ z\in \C; \  \Re(z)<1-2\beta \text{ and } \vert \mathrm{Im}(z)\vert< \beta/2\}\subset \mathcal{E}$ for all $\beta>0$.
\begin{proposition}\label{prop cv unif martingale}
	There exists $\beta_0>0$ such that the following holds for all $\beta\in (0, \beta_0)$. The sequence of functions $(z\mapsto M_n(z))_{n\ge n_0}$ converges uniformly almost surely and in $\mathrm{L}^1$ towards a random analytic function $z \mapsto M_\infty(z)$ on every compact subset of $\mathcal{E}_\beta$. Moreover, for any compact subset $K \subset \mathcal{E}_\beta$, there exists $\vp>0$ such that almost surely,
	\[
	\limsup_{n\to \infty} \max_{z \in K} n^\vp \vert M_n(z)-M_\infty(z)\vert <\infty.
	\]
	Furthermore, for any compact subset $K \subset \mathcal{E}_\beta$, for all $p>1$ small enough,
	\begin{equation}\label{eq M n borne dans L p}
	\sup_{n\ge n_0} \sup_{z \in K} \E\left( \vert M_n (z) \vert^p \right)<\infty.
	\end{equation}
\end{proposition}
\begin{proof}
	Recall that by construction, for all $n\ge n_0$ the function $z\mapsto M_n(z)$ is analytic on $\mathcal{E}$ and that it is measurable with respect to $\mathcal{F}_n$ (where we recall that $(\mathcal{F}_n)_{n\ge 0}$ is the natural filtration associated with $(\T_n)_{n\ge 0}$). Furthermore, recall that for all $z \in \mathcal{E}$, the sequence $(M_n(z))_{n\ge n_0}$ is a martingale with respect to $(\mathcal{F}_n)_{n\ge n_0}$.
	
	Recall from Proposition \ref{prop controle increments martingale} that for all compact $K\subset \mathcal{E}$, for all $p \in (1,2)$, for all $\beta \in (0,2-p)$,
	\[
	\E\left( \left\vert M_{2n}(z)- M_n(z)\right\vert^p\right)\le 
	\exp\left( \left(e^{p \Re(z) +(p-1)\beta}-p\Re(e^z)\right) \log (n)/c + O(1)\right),
	\]
	uniformly in $z \in K$ as $n\to \infty$, where the $O(1)$ is a sequence bounded by a constant which only depends on $K, p, \beta$ and $f$.

	Let $\beta \in (0,1/2)$. Take $p=1+\beta/2$. Then, for all $z \in \mathcal{E}_\beta$, 
	\begin{align*}
		e^{p\Re(z)+(p-1)\beta} - p \Re(e^z) &=
		e^{(1+\beta/2) \Re(z) + \beta^2/2}- \left(1+\frac{\beta}{2}\right) e^{\Re(z)} \cos(\mathrm{Im}(z))\\
		&\le e^{\Re(z)}\left( e^{(\beta/2)(1-2\beta) + \beta^2/2} - \left(1 + \frac{\beta}{2}\right)\left(1-\frac{\beta^2}{8}\right) \right).
	\end{align*}
	But, as $\beta \downarrow 0$, we have
	\[
	e^{(\beta/2)(1-2\beta) + \beta^2/2} - \left(1 + \frac{\beta}{2}\right)\left(1-\frac{\beta^2}{8}\right) = -\beta^2 +\frac{\beta^2}{2} + \frac{\beta^2}{8} + \frac{\beta^2}{8} + o(\beta^2),
	\]
	which is strictly negative for $\beta<\beta_0$ for some $\beta_0\in (0,1/2)$.
	
	One can then apply Lemma \ref{lemme 37} \eqref{item ii lemma 37} to the sequence of functions $(z\mapsto M_n(z))_{n\ge n_0}$ with $q=p=1+\beta/2$, with $\alpha(z)=(e^{p\Re(z) +(p-1)\beta}-p\Re(e^{z}))/(2cp)<0$ and $\delta(z)=-(e^{p\Re(z) +(p-1)\beta}-p\Re(e^{z}))/(2c) >0$ for all $z \in \mathcal{E}_\beta$, with $\beta \in (0, \beta_0)$. This proves the statement of the proposition, except \eqref{eq M n borne dans L p}.
	
	Let us show \eqref{eq M n borne dans L p}. By our choice of $\beta_0 $ we have for all $\beta \in (0, \beta_0)$, for all $K\subset \mathcal{E}_\beta$, letting $p=1+\beta /2$, there exists $\delta>0$ such that for all $z \in K$, 
	\[
	\E\left( \vert M_{2n}(z)- M_n(z) \vert^p\right)\le n^{-\delta },
	\]
	so that using \eqref{eq lemme Biggins},
	\[
	\sup_{n\ge \log(n_0)/\log(2)} \sup_{z \in K} \E\left( \vert M_{2^n}(z) \vert^p\right)<\infty.
	\]
	But actually, by the second point of Proposition \ref{prop controle increments martingale}, we see that for all $z \in K$, for all $n\ge n_0$, for all $k \in \lb 0 , n \rb$,
	\[
	\E\left( \vert M_{n+k}(z)- M_n(z) \vert^p\right)\le n^{-\delta },
	\]
	so that for all $k \in \lb 0, 2^n\rb$,
	\[
	\E\left( \vert M_{2^n+k}(z)- M_{2^n}(z) \vert^p\right)\le 2^{-\delta n}.
	\]
	Thus,
	\[
	\sup_{n\ge n_0}\sup_{z \in K} \E\left(\vert M_n(z)\vert^p \right)<\infty.
	\]
	This concludes the proof.
\end{proof}
\subsection{Zeros of the limit}\label{sous-section zeros de la limite}
We study the limits of the martingales $(M_n(z))_{n\ge n_0}$ for $z \in (-\infty, 1)$. Since in our case the weights depend on the depth, the self-similarity argument used in \cite{Sen21, KKS25} cannot be applied. Instead, we will control $\log M_n(0)$. We will only prove that almost surely, $z \mapsto M_\infty(z)$ has at most countably many zeros which are isolated. We start by proving a few lemmas. 

The first lemma gives rough bounds for the depth of the tree.
\begin{lemma}\label{lemme controle hauteur totale}
	For all $\vp \in (0,1)$, the sequence of random variables $e^{(1-\vp)d(\T_n)}/n^{e/c}$ is bounded in $\mathrm{L}^1$. In particular, almost surely, for all $n$ large enough, $d(\T_n)\le (3+2e/c)\log n$.
\end{lemma}
\begin{proof}
	Let $z \in (0,1)$. By definition of $\cL_n(z)$, we have
	\[
	e^{zd(\T_n)- S_{d(\T_n)}} \le \cL_n(z) =M_n(z)C_n(z)\le M_n(z) \exp(e^z \mathrm{h}_n+O(1)) \le M_n(z) \exp((e^z/c)\log(n) + O(1)),
	\]
	as $n\to \infty$, after applying Lemma \ref{lemme equivalent C n} and \eqref{eq encadrement h n}, where the $O(1)$ is deterministic.
	
	Moreover, by \eqref{condition ergodicite}, for all $\vp \in (0,1)$, there exists $c_\vp>0$ such that for all $n\ge1$, we have $e^{-S_{d(\T_n)}} \ge c_\vp e^{-\vp d(\T_n)}$. By taking the expectation, we conclude that as $n\to \infty$,
	\[
	c_\vp \E\left( e^{(z-\vp) d(\T_n)}\right) \le \E\left(M_{n_0}(z)\right) n^{e^z/c +O(1)} = o(n^{e/c}).
	\]
	This proves the first part of the lemma. Using that $\P(d(\T_n) \ge (3+2e/c)\log n) = \P(e^{d(\T_n)/2} \ge n^{3/2+e/c})$, Markov's inequality and then the Borel-Cantelli lemma, we deduce the second part of the lemma.
\end{proof}

The above lemma enables to get an almost sure equivalent to $\log (\cL_n(0)) $.
\begin{corollary}\label{cor controle transformee de Laplace}
	Almost surely, $ \log (\cL_n(0)) \sim \log n$ as $n\to \infty$.
\end{corollary}
\begin{proof}
	Let $\vp>0$. By Lemma \ref{lemme controle hauteur totale}, we know that a.s.\@ $d(\T_n)\le (3+2e/c)\log n$ for all $n$ large enough. Moreover, by \eqref{condition ergodicite}, there exists $C_\vp >0$ such that for all $d \ge 0$,
	\[
	\frac{1}{C_\vp} e^{-\vp d} \le e^{-S_d} \le C_\vp e^{\vp d}.
	\]
	Thus, we get that almost surely, for all $n$ large enough,
	\begin{multline*}
	\frac{1}{C_\vp} (n+1)e^{-\vp(3+2e/c)\log n }\le \sum_{v \in \T_n} \frac{1}{C_\vp} e^{-\vp d(\T_n)} \le \cL_n(0)\le \sum_{v \in \T_n} C_\vp e^{\vp d(\T_n)} \le (n+1) C_\vp e^{\vp(3+2e/c)\log n }.
	\end{multline*}
	Thus, for all $\vp>0$, almost surely, for all $n$ large enough, $(1-\vp) \log n \le \log (\cL_n(0)) \le (1+\vp) \log n$. This ends the proof.
\end{proof}

Equipped with the above lemmas, let us show the final result of the subsection.
\begin{lemma}\label{lemme limite non nulle}
	Almost surely, $M_\infty(0)> 0$. In particular, a.s.\@ the set of $z \in (-\infty, 1)$ such that $M_\infty(z) =0$ is at most countable and its elements are isolated.
\end{lemma}
\begin{proof}
	It suffices to prove that $M_\infty(0)> 0$ almost surely since we know by Proposition \ref{prop cv unif martingale} that the function $z \mapsto M_\infty(z)$ is analytic. For all $n\ge n_0$, let $\mathcal{A}_n$ be the event that we have $d(\T_n)\le (3+2e/c)\log n$ and $\cL_n(0) \ge n^{3/4}$. For all $k\ge n_0$, let $\mathcal{B}_k \coloneqq \bigcap_{n\ge k} \mathcal{A}_n$. Note that by Lemma \ref{lemme controle hauteur totale} and Corollary \ref{cor controle transformee de Laplace}, we know that $\bigcup_{k\ge n_0}\mathcal{B}_k$ has full probability.
	
	Let $k\ge n_0$. Let us show that the sequence $(u_n)_{n\ge k}$ defined by
	\[
	u_n\coloneqq \E\left(\vert \log M_n(0)\vert^2 {\bf 1}_{\bigcap_{i=k}^n \mathcal{A}_i}\right)
	\]
	is bounded. Note that $M_n(0)>0$ for all $n\ge n_0$ so that the above display is well-defined.
	
	Let us first compute
	\begin{align*}
		\E&\left( \left. (\log M_{n+1}(0))^2{\bf 1}_{\bigcap_{i=k}^{n+1} \mathcal{A}_i} \right\vert \mathcal{F}_n \right)
		\le \E\left( \left. (\log M_{n+1}(0))^2{\bf 1}_{\bigcap_{i=k}^{n} \mathcal{A}_i} \right\vert \mathcal{F}_n \right) \\
		=&{\bf 1}_{\bigcap_{i=k}^{n} \mathcal{A}_i}
		\sum_{v \in \T_n} \frac{f(d(v))}{Z_n} \left( \log\left(\cL_n(0)+ e^{-S_{d(v)+1}}\right) -\log C_{n+1}(0)\right)^2\\
		=&{\bf 1}_{\bigcap_{i=k}^{n} \mathcal{A}_i}\sum_{v \in \T_n} \frac{f(d(v))}{Z_n} \left(\log(M_n(0)) + \log \left(1+ \frac{e^{-S_{d(v)+1}} }{\cL_n(0)}\right) -\log\left(1+\frac{1}{Z_n}\right)\right)^2\\
		=&{\bf 1}_{\bigcap_{i=k}^{n} \mathcal{A}_i} \left( (\log M_n(0))^2 + \sum_{v \in \T_n} \frac{f(d(v))}{Z_n} \left(\log \left(1+ \frac{e^{-S_{d(v)+1}} }{\cL_n(0)}\right) -\log\left(1+\frac{1}{Z_n}\right)\right)^2 \right. \\
		&\left.+ \sum_{v \in \T_n} \frac{f(d(v))}{Z_n}2(\log M_n(0))\left(\log \left(1+ \frac{e^{-S_{d(v)+1}} }{\cL_n(0)}\right) -\log\left(1+\frac{1}{Z_n}\right)\right)  \right).
	\end{align*}
	Moreover, on the event $\mathcal{A}_n$, the first sum in the last equality can be bounded from above as follows. For all $\vp>0$, there is a constant $C_\vp>0$ such that for all $n\ge k$, on the event $\mathcal{A}_n$,
	\begin{align}
		\sum_{v \in \T_n} \frac{f(d(v))}{Z_n} \left(\log \left(1+ \frac{e^{-S_{d(v)+1}} }{\cL_n(0)}\right) -\log\left(1+\frac{1}{Z_n}\right)\right)^2&\le
		\sum_{v \in \T_n} \frac{f(d(v))}{Z_n}C_\vp \left( \frac{e^{2\vp (d(\T_n)+1)}}{n^{3/2}} + \frac{1}{n^2} \right) \notag \\
		&= C_\vp \left( \frac{e^{2\vp (1+(3+2e/c)\log n)}}{n^{3/2}} + \frac{1}{n^2} \right), \label{eq premiere somme}
	\end{align}
	where in the inequality, we used \eqref{condition ergodicite} to control $S_{d(v)+1}$ and we used that $Z_n \ge cn$ since $f\ge c$.
	
	For the last sum, one can note that
	\begin{multline}\label{eq derniere somme}
		 \left\vert\sum_{v \in \T_n} \frac{f(d(v))}{Z_n}2(\log M_n(0))\left(\log \left(1+ \frac{e^{-S_{d(v)+1}} }{\cL_n(0)}\right) -\log\left(1+\frac{1}{Z_n}\right)\right) \right\vert \\ 
		 \le (1+ (\log M_n(0))^2 ) \left\vert \sum_{v \in \T_n} \frac{f(d(v))}{Z_n}\left(  \log \left(1+ \frac{e^{-S_{d(v)+1}} }{\cL_n(0)}\right) -\log\left(1+\frac{1}{Z_n}\right)\right) \right\vert 
	\end{multline}
	On the event $\mathcal{A}_n$ the sum in the absolute value on the right-hand side of \eqref{eq derniere somme} can be bounded from above by
	\[
	\sum_{v \in \T_n} \frac{f(d(v))}{Z_n} \left( \frac{e^{-S_{d(v)+1}}}{\cL_n(0)} - \log \left(1+\frac{1}{Z_n}\right)\right)= \frac{1}{Z_n} - \log \left(1+ \frac{1}{Z_n}\right) \le \frac{C}{n^2},
	\]
	for some constant $C>0$ (using that $Z_n \ge cn$ once more). Moreover, using that $\log(1+x) \ge x-x^2/2$ for all $x\ge 0$, for all $\vp>0$, there is a constant $\widetilde{C}_\vp>0$ such that for all $n\ge k$, on the event $\mathcal{A}_n$, the sum in the absolute value on the right-hand side of \eqref{eq derniere somme} can be bounded from below by
	\begin{align*}
	\sum_{v \in \T_n} &\frac{f(d(v))}{Z_n}\left( \frac{e^{-S_{d(v)+1}}}{\cL_n(0)} - \left(\frac{e^{-S_{d(v)+1}}}{\cL_n(0)} \right)^2 \right) - \log \left(1+\frac{1}{Z_n}\right)\\
	&\ge  \frac{1}{Z_n} - \log \left(1+ \frac{1}{Z_n}\right)  - \sum_{v \in \T_n} \frac{f(d(v))}{Z_n} \widetilde{C}_\vp \frac{e^{2\vp (d(\T_n)+1)}}{n^{3/2}} \\
	&\ge - \widetilde{C}_\vp \frac{e^{2\vp(1+(3+2e/c)\log n)}}{n^{3/2}}.
	\end{align*}
	Thus, for all $\vp>0$, there exists $C'_\vp>0$ such that for all $n\ge k$, on the event $\mathcal{A}_n$, 
	\begin{multline}\label{eq derniere somme 2}
		\left\vert\sum_{v \in \T_n} \frac{f(d(v))}{Z_n}2(\log M_n(0))\left(\log \left(1+ \frac{e^{-S_{d(v)+1}} }{\cL_n(0)}\right) -\log\left(1+\frac{1}{Z_n}\right)\right) \right\vert \\
		\le C'_\vp (1+ (\log M_n(0))^2) \left( \frac{e^{2\vp(1+(3+2e/c)\log n)}}{n^{3/2}} +  \frac{1}{n^2} \right).
	\end{multline}
	Combining \eqref{eq derniere somme 2} with \eqref{eq premiere somme}, and taking $\vp>0$ small enough, we deduce that there exists a constant $K>0$ such that for all $n\ge k$,
	\[
	\E\left(\left. (\log M_{n+1}(0))^2 {\bf 1}_{\bigcap_{i=k}^{n+1} \mathcal{A}_i} \right\vert \mathcal{F}_n\right)\le {\bf 1}_{\bigcap_{i=k}^n \mathcal{A}_i} \left( \left(1+ \frac{K}{n^{5/4}} \right) (\log M_n(0))^2 + \frac{K}{n^{5/4}} \right).
	\]
	Taking the expectation, we deduce that for all $n\ge k$,
	\[
	u_{n+1} \le \left(1 + \frac{K}{n^{5/4}}\right) u_n + \frac{K}{n^{5/4}}.
	\]
	One can then easily see that $(u_n)_{n\ge k}$ is bounded (letting $v_n = 1+u_n$, we see that $v_{n+1} \le (1+K/n^{5/4})v_n$).
	
	Recall that $\mathcal{B}_k = \bigcap_{n\ge k} \mathcal{A}_n$. Note that for all $n\ge k$,
	\[
	\E\left((\log M_n(0))^2 {\bf 1}_{\mathcal{B}_k} \right)\le u_n.
	\]
	In particular, by Fatou's lemma, since $M_n(0)\to M_\infty(0)$ almost surely by Proposition \ref{prop cv unif martingale}, we conclude that
	\[
	\E\left((\log M_\infty(0))^2  {\bf 1}_{\mathcal{B}_k} \right) \le \liminf_{n\to \infty}\E\left((\log M_n(0))^2 {\bf 1}_{\mathcal{B}_k} \right) <\infty. 
	\]
	Thus, almost surely, $M_\infty(0) >0$ on the event $\mathcal{B}_k$. Letting $k\to \infty$ proves that $M_\infty(0)>0$ almost surely. Finally, recall from Proposition \ref{prop cv unif martingale} that $z\mapsto M_\infty(z)$ is a random analytic function (as a uniform limit of analytic functions). Therefore, since the zeros of a non-zero analytic function are at most countable and isolated, we conclude that almost surely, on the interval $(-\infty, 1)$, the function $z\mapsto M_\infty(z)$ has at most countably many zeros that are isolated.
\end{proof}
\subsection{Concentration}\label{sous-section concentration}
This subsection is devoted to the proof of the following concentration bound.
\begin{lemma}\label{lemme controle queue integrale}
	For every compact subset $K \subset (-\infty, 1)$, for all $a>0$, there exists $\vp>0$ such that almost surely,
	\[
	\sup_{\beta \in K} e^{-e^\beta \mathrm{h}_n} \int_a^\pi \left\vert \cL_n(\beta+iu) \right\vert du  = O\left(\frac{1}{n^\vp}\right).
	\]
\end{lemma}

In order to prove the above lemma, we need a more precise control on the martingales $M_n(z)$. We start by proving a variant of Proposition \ref{prop controle increments martingale}.
\begin{lemma}\label{lemme variante controle increments martingale}
	Let $K \subset \mathcal{E}$ be a compact subset. Then there exist $p_0 \in (1,2)$ and $\vp_0 \in (0, 2-p_0)$ such that for all $p \in (1,p_0)$, for all $\vp \in (0,2-p)\cap (0, \vp_0)$, almost surely,
	\[
 \sup_{n\ge n_0} \sup_{z \in K}\left( \E\left( \left. \left\vert M_{2n}(z)- M_n(z)\right\vert^p \right\vert \mathcal{F}_n\right) \times \exp\left( -\left(e^{p \Re(z) +(p-1)\vp}-p\Re(e^z)\right) \mathrm{h}_n \right)\right) <\infty.
	\] 
	Moreover, for $q>1$ close enough to $1$, 
	\[
	\sup_{n\ge n_0} \sup_{z \in K}  \E\left(  \left( \E\left( \left. \left\vert M_{2n}(z)- M_n(z)\right\vert^p \right\vert \mathcal{F}_n\right) \times \exp\left( -\left(e^{p \Re(z) +(p-1)\vp}-p\Re(e^z)\right) \mathrm{h}_n \right)\right)^q \right) <\infty.
	\] 
\end{lemma}
\begin{proof}
	Let $p_0 \in (1,2)$ and $\vp_0 \in (0, 2-p_0)$ such that $p_0 \Re(z)+ (p_0-1)\vp_0 <1$ for all $z \in K$. Let $p \in (1,p_0)$ and $\vp \in (0,2-p)\cap (0, \vp_0)$.
	
	By \eqref{eq majoration moment conditionnel increment} in the proof of Proposition \ref{prop controle increments martingale}, we know that there exist constants $c(K,p)>0$ and $c(\vp, K, p)>0$ such that for all $n\ge n_0$, for all $z \in K$,
	\[
	\E\left( \left. \left\vert M_{n+1}(z)-M_n(z) \right\vert^p  \right\vert \mathcal{F}_n\right)\le c(K,p) \frac{1}{n^p} \vert M_n(z) \vert^p + \frac{c(\vp, K, p)}{n} \frac{\cL_n(p\Re(z)+(p-1)\vp)}{\vert C_{n}(z)\vert^p}.
	\]
	Moreover, by Lemma \ref{lemme moment Lp martingale}, we know that
	\[
	\vert M_n(z)\vert^p\le C_\vp \frac{n^{p-1} \cL_n(p \Re(z)+(p-1)\vp)}{\vert C_n(z)\vert^p}.
	\]
	Therefore, there is a constant $c'(\vp, K, p)$ such that for all $n\ge n_0$, for all $z \in K$,
	\begin{align*}
	\E\left( \left. \left\vert M_{n+1}(z)-M_n(z) \right\vert^p  \right\vert \mathcal{F}_n\right)&\le \frac{c'(\vp, K, p)}{n} \frac{\cL_n(p\Re(z)+(p-1)\vp)}{\vert C_{n}(z)\vert^p} \\
	&= \frac{c'(\vp, K, p)}{n} \frac{C_n(p\Re(z)+(p-1)\vp)}{\vert C_n(z)\vert^p} M_n(p\Re(z)+(p-1)\vp).
	\end{align*}
	Next, by Lemma \ref{lemme equivalent C n}, for all $n\ge n_0$, for all $z \in K$,
	\begin{align*}
		\E&\left( \left. \left\vert M_{n+1}(z)-M_n(z) \right\vert^p  \right\vert \mathcal{F}_n\right)\\&\le 
	\frac{c'(\vp, K, p)}{n} \exp\left( \left(e^{p \Re(z) +(p-1)\vp}-p\Re(e^z)\right) \mathrm{h}_n +c''(\vp, K, p)\right)
	M_n(p\Re(z)+(p-1)\vp),
	\end{align*}
	where $c''(\vp, K, p)>0$ is a constant. Note that for all $k \in \lb 0, n \rb$, we have $0\le \mathrm{h}_{n+k} -\mathrm{h}_{n} \le \sum_{j=n+1}^{n+k} 1/Z_j \le \sum_{j=n+1}^{2n} 1/(cj)\le 1/c$. Therefore, for all $n\ge n_0$, for all $k \in \lb 0, n-1 \rb$, for all $z \in K$,
	\begin{align*}
		\E&\left( \left. \left\vert M_{n+k+1}(z)-M_{n+k}(z) \right\vert^p  \right\vert \mathcal{F}_{n+k}\right)\\&\le 
		\frac{c'(\vp, K, p)}{n} \exp\left( \left(e^{p \Re(z) +(p-1)\vp}-p\Re(e^z)\right) \mathrm{h}_n +c'''(\vp, K, p)\right)
		M_{n+k}(p\Re(z)+(p-1)\vp),
	\end{align*}
	where $c'''(\vp, K, p)>0$ is a constant. Taking the expectation conditionally on $\mathcal{F}_n$, we obtain for all $n\ge n_0$, for all $k \in \lb 0, n-1 \rb$, for all $z \in K$,
	\begin{align*}
		\E&\left( \left. \left\vert M_{n+k+1}(z)-M_{n+k}(z) \right\vert^p  \right\vert \mathcal{F}_{n}\right)\\&\le 
		\frac{c'(\vp, K, p)}{n} \exp\left( \left(e^{p \Re(z) +(p-1)\vp}-p\Re(e^z)\right) \mathrm{h}_n +c'''(\vp, K, p)\right)
		M_{n}(p\Re(z)+(p-1)\vp).
	\end{align*}
	Finally, by \eqref{eq lemme Biggins}, we conclude that for all $n\ge n_0$, for all $z \in K$,
	\begin{align*}
		\E&\left( \left. \left\vert M_{2n}(z)-M_{n}(z) \right\vert^p  \right\vert \mathcal{F}_{n}\right)\\&\le 
		{2^p c'(\vp, K, p)} \exp\left( \left(e^{p \Re(z) +(p-1)\vp}-p\Re(e^z)\right) \mathrm{h}_n +c'''(\vp, K, p)\right)
		M_{n}(p\Re(z)+(p-1)\vp).
	\end{align*}
	By Proposition \ref{prop cv unif martingale} and because $\max_{z \in K} p \Re(z) + (p-1)\vp <1$, the sequence of functions $(z\mapsto M_{n}(p\Re(z)+(p-1)\vp))_{n\ge n_0}$ converges uniformly on $K$ almost surely, and is bounded in $\mathrm{L}^q$ uniformly in $z \in K$ for $q>1$ close enough to $1$, hence the desired result.
\end{proof}
We then state two lemmas on the behaviour of $\mathrm{h}_n$. The first one is about its exponential moments.
\begin{lemma}\label{lemme moment exponentiel h n}
	We have $\log (\E(e^{\mathrm{h}_n})) \sim \log n$ as $n\to \infty$. Moreover, for all $\vp>0$, for all $\alpha_0>0$, we have
	\[
	\sup_{n\ge n_0, \alpha \in [0, \alpha_0]} \frac{1}{n^{\alpha+\vp}} \E\left(e^{\alpha \mathrm{h}_n}\right) <\infty.
	\]
\end{lemma}
\begin{proof}
	By Lemma \ref{lemme equivalent C n}, we know that $C_n(0) = \exp (\mathrm{h}_n + O(1))$, where the $O(1)$ is bounded in absolute value by a deterministic constant. Moreover, by Lemma \ref{lemme martingale} and a simple induction, we have $\E(\cL_n(z)) = \E(C_n(z))$ for all $n\ge n_0$. Thus,
	\begin{equation}\label{eq lien h n L n}
	\E\left(\cL_n(0)\right) = \E\left( C_n(0) \right) = \E\left(e^{\mathrm{h}_n+ O(1)}\right),
	\end{equation}
	where the $O(1)$ is bounded in absolute value by a deterministic constant.
	
	Furthermore, by Lemma \ref{lemme controle hauteur totale}, the sequence $e^{d(\T_n)/2}/n^{e/c}$ is bounded in $\mathrm{L}^1$. In particular, for all $\vp\in (0,1/2)$, for all $n\ge n_0$,
	\begin{equation}\label{eq majoration moment exponentiel depth}
	\E\left( e^{\vp d(\T_n)} \right) \le \E\left( e^{\vp d(\T_n)/(2\vp)} \right)^{2\vp} = \E\left(e^{d(\T_n)/2}\right)^{2\vp} = O\left( n^{2\vp e/c} \right),
	\end{equation}
	as $n\to \infty$. Recall that by \eqref{condition ergodicite}, there is a constant $C_\vp>0$ such that $e^{-S_d}\le C_\vp e^{\vp d}$ for all $d \ge 0$. Thus, for all $n\ge 0$,
	\begin{equation}\label{eq majoration L n}
		\cL_n(0) \le (n+1) C_\vp e^{\vp d(\T_n)}.
	\end{equation}
	Therefore, as $n\to \infty$, we have
	\[
	\E\left(\cL_n(0)\right)\le \E\left((n+1) C_\vp e^{\vp d(\T_n)} \right) = O\left(n^{1+2\vp e/c}\right).
	\]
	By \eqref{eq lien h n L n}, this proves that for all $\vp>0$,
	\[
	\limsup_{n\to \infty} \frac{\log(\E(e^{\mathrm{h}_n}))}{\log n} \le 1+ \vp.
	\]
	For the lower bound, recall from Corollary \ref{cor controle transformee de Laplace} that almost surely, $ \log (\cL_n(0)) \sim \log n$ as $n\to \infty$. In particular, for all $\vp>0$, for all $n$ large enough, with probability at least $1/2$, we have $\cL_n(0) \ge n^{1-\vp}$. Thus, for all $n$ large enough,
	\[
	\E\left(\cL_n(0)\right) \ge \frac{1}{2} n^{1-\vp}.
	\]
	By \eqref{eq lien h n L n}, this proves that for all $\vp>0$,
	\[
	\liminf_{n\to \infty} \frac{\log(\E(e^{\mathrm{h}_n}))}{\log n} \ge 1-\vp,
	\]
	thus ending the proof of the first part of the lemma. Moreover, by Hölder's inequality $\E(e^{\alpha\mathrm{h}_n}) \le \E(e^{\mathrm{h}_n})^\alpha$ for all $\alpha \in [0,1]$ so for the second part, it suffices to focus on the case $\alpha\ge 1$.
	
Let us deal with this remaining case. Let $\alpha_0>1$. For all $\alpha\ge 1$ and $n\ge n_0$, 
\begin{align*}\E(\cL_{n+1}(0)^\alpha \vert \mathcal{F}_n) &=\sum_{v \in \T_n}  \frac{f(d(v))}{Z_n} \left( \cL_n(0) + e^{-S_{d(v)+1}} \right)^\alpha  \\&= \cL_n(0)^\alpha \sum_{v \in \T_n}  \frac{f(d(v))}{Z_n} \left( 1 + \frac{e^{-S_{d(v)+1}} }{\cL_n(0)}\right)^\alpha
	\\
	&\ge \cL_n(0)^\alpha \sum_{v \in \T_n} \frac{f(d(v))}{Z_n} \left( 1 +\alpha \frac{e^{-S_{d(v)+1}} }{\cL_n(0)}\right)\\
	&= \cL_n(0)^\alpha  \left( 1+ \alpha \frac{1}{Z_n} \right).
 \end{align*}
 So, by induction,
 \[
 \E(\cL_n(0)^\alpha) \ge \E \left(C_n(\log \alpha) \cL_{n_0}(0)^\alpha\right).
 \]
 By Lemma \ref{lemme equivalent C n}, we know that $C_n(\log \alpha) = \exp (\alpha\mathrm{h}_n + O(1))$ uniformly in $\alpha \in [1, \alpha_0]$, where the $O(1)$ is bounded in absolute value by a deterministic constant. Moreover, note that $1 \le \cL_{n_0}(0)^\alpha$, so that there exists a constant $c_{\alpha_0}>0$ such that for all $n\ge n_0$, for all $\alpha \in [1, \alpha_0]$,
 \[
 \E\left(e^{\alpha\mathrm{h}_n} \right) \le  c_{\alpha_0}\E(\cL_n(0)^\alpha).
 \]
 Thus, it suffices to show that for all $\vp>0$, for all $\alpha_0>1$,
 \[
 \sup_{\alpha \in [1, \alpha_0], n\ge n_0} \frac{1}{n^{\alpha+\vp}} \E\left(\cL_n(0)^\alpha \right) <\infty.
 \]
 We proceed as in the first part of the proof. By \eqref{eq majoration L n}, we know that for all $\vp'>0$, there is $C_{\vp'}>1$ such that for all $n\ge n_0$, for all $\alpha \in [1, \alpha_0]$,
 \[
 \E\left( \cL_n(0)^\alpha \right) \le (n+1)^\alpha C_{\vp'}^\alpha \E\left(e^{\alpha \vp' d(\T_n)} \right)\le (n+1)^\alpha C_{\vp'}^{\alpha_0} \E\left(e^{\alpha_0 \vp' d(\T_n)} \right).
 \]
 We conclude by taking $\vp'$ small enough and by applying \eqref{eq majoration moment exponentiel depth}.
\end{proof}
The next lemma gives an equivalent for $\mathrm{h}_n$.
\begin{lemma}\label{lemme equivalent ps h n}
	We have $\mathrm{h}_n \sim \log n$ almost surely as $n\to \infty$. More precisely, we have $\mathrm{h}_n= \log n + O(\max_{0 \le k \le d(\T_n)}\vert S_k\vert)$ almost surely.
\end{lemma}
\begin{proof}
	Recall from Lemma \ref{lemme equivalent C n} that $C_n(0) = \exp(\mathrm{h}_n +O(1))$ a.s.\@ as $n\to \infty$. Moreover, we know that the non-negative martingale $M_n(0)$ converges to $M_\infty(0)$ given by Proposition \ref{prop cv unif martingale}. By Lemma \ref{lemme limite non nulle}, almost surely $M_\infty(0) > 0$. Recalling that $M_n(0) = \cL_n(0)/C_n(0)$, we deduce that a.s.
	\[
	\log \cL_n(0) \mathop{\sim}\limits_{n\to \infty} \log C_n(0).
	\]
	Finally, by Corollary \ref{cor controle transformee de Laplace}, we know that $\log \cL_n(0) \sim \log n$ almost surely as $n\to \infty$, yielding the first part of the lemma.
	
	For the second part of the lemma, one can see that more precisely, we have
	\[
	(n+1)e^{-\max_{0 \le k\le d(\T_n)} \vert S_k\vert }\le \cL_n(0) \le (n+1)e^{\max_{0 \le k\le d(\T_n)} \vert S_k\vert },
	\]
	so that, recalling from above that $\cL_n(0)/C_n(0) = M_n(0) \to M_\infty(0)$, we obtain the desired result.
\end{proof}
Using the above lemmas we are able to get the following uniform control of the martingale away from the real axis.
\begin{lemma}\label{lemme controle uniforme martingale queue}
	Let $[b_1,b_2] \subset (-\infty, 1)$ be a segment and let $a \in (0,\pi)$. Then, there exists $\vp>0$ such that almost surely, as $n\to \infty$,
	\[ \sup_{\beta \in [b_1, b_2], u \in [a, \pi]} 
	e^{-(1-\cos(u))e^\beta \mathrm{h}_n} \left\vert M_n(\beta +iu)-M_{n_0}(\beta +iu) \right\vert  = O\left( \frac{1}{n^\vp}\right).
	\]
\end{lemma}
\begin{proof}
	For all $p \in (1,2)$, $\vp \in (0, 2-p)$, for all $z\in \mathcal{E}$, for all $n\ge n_0$, let
	\[
	Z^{\vp, p}_n(z) \coloneqq
	 \E\left( \left. \left\vert M_{2n}(z)- M_n(z)\right\vert^p \right\vert \mathcal{F}_n\right) \times \exp\left( -\left(e^{p \Re(z) +(p-1)\vp}-p\Re(e^z)\right) \mathrm{h}_n \right) .
	\]
	Let $a \in (0, \pi)$ and $b_0<0$. Let $\mathcal{O} \coloneqq \{ \beta+ i u; \ b_0< \beta <1 \text{ and } u \in (a/2, \pi+1)\} \subset \mathcal{E}$ and let $K \subset \mathcal{O}$ be a compact subset. By Lemma \ref{lemme variante controle increments martingale}, we know that there exists $q>1$ such that $\sup_{n\ge n_0} \sup_{z \in K}\E(\vert Z^{\vp, p}_n(z) \vert^q)<\infty$.

	Note that for all $z \in \mathcal{E}$, for all $n\ge n_0$,
	\begin{equation}\label{eq M 2n moins M n egal Z}
	\E\left( \left\vert M_{2n}(z)- M_n(z)\right\vert^p \right) =\E\left(Z^{\vp, p}_n(z)  \exp\left( \left(e^{p \Re(z) +(p-1)\vp}-p\Re(e^z)\right) \mathrm{h}_n \right) \right).
	\end{equation}
	Moreover, for all $\beta\in [b_0, 1]$, for all $u \in [a/2, \pi+1]$, if $z= \beta+ iu$, then
	\begin{align}
		e^{p \Re(z) +(p-1)\vp}-p\Re(e^z) =  e^{p \beta + (p-1) \vp}  - p e^\beta \cos(u)
		&\mathop{\longrightarrow}\limits_{\vp \to 0}  e^{p \beta} - pe^\beta \cos(u)\notag
		\\
		&\mathop{\longrightarrow}\limits_{p \downarrow 1}  (1- \cos(u)) e^\beta, \label{eq limite p et epsilon vers zero}
	\end{align}
	where the convergences are uniform in $\beta\in [b_0, 1]$ and $u \in [a/2, \pi+1]$. Therefore, by choosing $p>1$ close enough to $1$ and $\vp>0$ small enough, we have
	\begin{equation}\label{eq minimum strictement positif}
	\inf_{z \in \mathcal{O}} \left( e^{p \Re(z) +(p-1)\vp}-p\Re(e^z) \right) \ge \frac{1}{2}\left(1- \cos\left(\frac{a}{2}\right)\right) e^{b_0}>0.
	\end{equation}
	Let $q>1$ such that $\sup_{n\ge n_0} \sup_{z \in K}\E(\vert Z^{\vp, p}_n(z) \vert^q)<\infty$ and let $q'>1$ such that $1/q+1/q'=1$. Then, by Hölder's inequality,
	\begin{multline}\label{eq Holder Z}
		\E\left(Z^{\vp, p}_n(z)  \exp\left( \left(e^{p \Re(z) +(p-1)\vp}-p\Re(e^z)\right) \mathrm{h}_n \right) \right)\\\le \E(\vert Z^{\vp, p}_n(z)   \vert^q)^{1/q}  \E\left(
		 \exp\left( q'\left(e^{p \Re(z) +(p-1)\vp}-p\Re(e^z)\right) \mathrm{h}_n  \right) \right)^{1/q'}.
	\end{multline}
	Furthermore, by Lemma \ref{lemme moment exponentiel h n}, we see that for all $\vp'>0$, 
	\[
	\sup_{z \in K, n\ge n_0}
	\frac{1}{n^{q'\left(e^{p \Re(z) +(p-1)\vp}-p\Re(e^z)\right) + \vp'}}
	\E\left(
	\exp\left( q'\left(e^{p \Re(z) +(p-1)\vp}-p\Re(e^z)\right) \mathrm{h}_n  \right) \right)
<\infty.	\]
Therefore, by taking \eqref{eq M 2n moins M n egal Z} and \eqref{eq Holder Z} into account, we see that for all $\vp'>0$, 
\[
\sup_{z \in K, n \ge n_0}  \frac{1}{n^{\left(e^{p \Re(z) +(p-1)\vp}-p\Re(e^z)\right) +  \vp'/q'}} 	\E\left( \left\vert M_{2n}(z)- M_n(z)\right\vert^p \right) <\infty.
\]
Let $\vp''\in (0, (1- \cos(a/2)) e^{b_0}/2)$. We set for all $z \in \mathcal{O}$,
\[
\left\{ 
\begin{matrix*}[l]
	\alpha(z)&\coloneqq &(1-\cos(\mathrm{Im}(z)))\frac{e^{\Re(z)}}{p} + \vp''\\
	\delta(z) &\coloneqq &(1-\cos(\mathrm{Im}(z)))e^{\Re(z)} - \left(e^{p \Re(z) +(p-1)\vp}-p\Re(e^z)\right) - \frac{\vp'}{q'} + p\vp''.
\end{matrix*}
\right.
\]
By our choice of $\vp''$, the left-hand side of \eqref{eq minimum strictement positif} is larger than $\vp''$, so that $\alpha(z)>0$ for all $z \in \mathcal{O}$. Moreover, by choosing $\vp'$ small enough, $p$ small enough and $\vp$ small enough, using \eqref{eq limite p et epsilon vers zero}, we see that for all $z \in \mathcal{O}$, we have $\delta(z)>0$. Thus, by Lemma \ref{lemme 37} \eqref{item i lemme 37}, we see that for any compact set $K \subset \mathcal{O}$, there exists $\tilde{\vp}>0$ such that almost surely,
\[
\sup_{z \in K, n\ge n_0} \frac{1}{n^{\alpha(z)- \tilde{\vp}}}\left\vert M_n(z) - M_{n_0}(z) \right\vert  <\infty.
\]
We apply the above result to the compact set $K \coloneqq \{ \beta + iu; \ \beta \in [b_1, b_2], u \in [a, \pi]\}$. The desired statement then stems from the facts that $\vp''$ can be chosen arbitrarily small, that $p$ can be chosen arbitrarily close to $1$ and that Lemma \ref{lemme equivalent ps h n} gives that $\mathrm{h}_n \sim \log n$ almost surely.
\end{proof}
We are now in position to prove Lemma \ref{lemme controle queue integrale}.
\begin{proof}[Proof of Lemma \ref{lemme controle queue integrale}]
	Let $K \subset (-\infty, 1)$ be a compact set and let $a \in (0, \pi)$. Let $b_1<b_2<1$ such that $K\subset [b_1, b_2]$. Then, let us write for all $\beta \in [b_1, b_2]$,
	\begin{align*}
		\int_a^\pi \vert \cL_n (\beta + iu)\vert du&= \int_a^\pi \vert C_n(\beta+ iu) \vert \vert M_n(\beta + iu ) \vert du\\
		&= \int_a^\pi  \vert \exp(e^{\beta + iu} \mathrm{h}_n + O(1)) \vert \vert M_n(\beta + iu ) \vert du,
	\end{align*}
	as $n\to \infty$, where the $O(1)$ is uniform in $\beta \in [b_1, b_2] , u \in [a, \pi]$ by Lemma \ref{lemme equivalent C n}. As a result,
	\begin{equation}\label{eq integrale C n asymptotique}
		\int_a^\pi \vert \cL_n (\beta + iu)\vert du =  \int_a^\pi  \exp(e^{\beta}\cos(u) \mathrm{h}_n + O(1))  \vert M_n(\beta + iu ) \vert du,
	\end{equation}
	where the $O(1)$ is uniform in $\beta \in [b_1, b_2] , u \in [a, \pi]$. 
	
	Furthermore, by Lemma \ref{lemme controle uniforme martingale queue}, there exists $\vp>0$ such that almost surely,
	\[ \sup_{\beta \in [b_1, b_2], u \in [a, \pi]} 
	e^{-(1-\cos(u))e^\beta \mathrm{h}_n} \left\vert M_n(\beta +iu)-M_{n_0}(\beta +iu) \right\vert  = O\left( \frac{1}{n^\vp}\right).
	\]
	But notice that since $z \mapsto M_{n_0}(z)$ is continuous on $[b_1, b_2]\times [a, \pi]$, it is bounded. Thus, using the inequality $(1-\cos(u))e^\beta \ge (1- \cos(a))e^{b_1}$ and the fact that $\mathrm{h}_n \sim \log n$ a.s.\@ by Lemma \ref{lemme equivalent ps h n}, by possibly replacing $\vp$ by $\vp \wedge (1-\cos(a))e^{b_1}/2$, we get that almost surely,
	\[ \sup_{\beta \in [b_1, b_2], u \in [a, \pi]} 
	e^{-(1-\cos(u))e^\beta \mathrm{h}_n} \left\vert M_{n_0}(\beta +iu) \right\vert  = O\left( \frac{1}{n^\vp}\right).
	\]
	Hence, almost surely,
	\[
	 \sup_{\beta \in [b_1, b_2], u \in [a, \pi]} 
	e^{-(1-\cos(u))e^\beta \mathrm{h}_n} \left\vert M_n(\beta +iu) \right\vert  = O\left( \frac{1}{n^\vp}\right).
	\]
	By taking \eqref{eq integrale C n asymptotique} into account, we obtain that almost surely, as $n\to \infty$,
	\[
	\int_a^\pi \vert \cL_n (\beta + iu)\vert du =  O\left( \frac{1}{n^\vp}\right) \times \int_a^\pi  \exp(e^{\beta}\cos(u) \mathrm{h}_n +(1-\cos(u))e^\beta \mathrm{h}_n )   du = O\left(\frac{1}{n^\vp}\right) \times \exp\left(e^\beta \mathrm{h}_n\right),
	\]
	where the $O(1/n^\vp)$ is uniform in $\beta \in [b_1, b_2]$.
\end{proof}

\subsection{Proof of the main result}\label{sous-section preuve du profil}
Recall that for all $n, k \in \N$, we denote by $\mathbb{L}_n(k)$ the number of vertices of $\T_n$ at depth $k$. In this last subsection, we check that the weighted profile $(\widetilde{\mathbb{L}}_n(k))_{k\ge 0} \coloneqq (\mathbb{L}_n(k) e^{-S_k})_{k\ge 0}$ satisfies the assumptions A1--A4 of \cite{KMS17} to satisfy the general Edgeworth expansion on every open interval $(\beta_-, \beta_+) \subset (-\infty, 1)$ on the event that the random function $M_\infty$ from Proposition \ref{prop cv unif martingale} does not vanish on $(\beta_-, \beta_+)$. Let us state the precise result. Recall the definition of the analytic function $c(z)$ from Lemma \ref{lemme equivalent C n}.
\begin{theorem}\label{th edgeworth expansion}
	For all $\beta_-<\beta_+\le 1$, there exists $\eta>0$ such that on the event that $M_\infty$ does not vanish on $(\beta_-, \beta_+)$, the following holds almost surely:
	\begin{enumerate}[({A}1)]
		\item For all $\beta \in (\beta_-, \beta_+)$, for all $n\in \N$, we have $\cL_n(\beta)<\infty$;
		\item Locally uniformly on $ \{\beta + iu; \ \beta \in (\beta_-, \beta_+), u \in (-\eta, \eta)\}$, we have
		\[
		W_n(z) \coloneqq e^{-e^z \mathrm{h}_n} \cL_n(z) \mathop{\longrightarrow}\limits_{n\to \infty} W_\infty(z) \coloneqq e^{c(z)} M_\infty(z),
		\]
		and the random analytic function $z \mapsto e^{c(z)} M_\infty(z)$ does not vanish on $(\beta_-, \beta_+)$;
		\item For every compact subset $K \subset \{\beta + iu; \ \beta \in (\beta_-, \beta_+), u \in (-\eta, \eta)\}$, there exists a finite random variable $C_K$ and $\vp>0$ such that 
		\[
		\sup_{z \in K} \left\vert W_n(z)- W_\infty(z)  \right\vert \le \frac{C_K}{n^\vp}
		\]
		and in particular the above expression is a $o({1}/{\mathrm{h}_n^r})$ as $n\to \infty$ for all $r \ge 0$;
		\item For every compact subset $K \subset (\beta_-, \beta_+)$, for all $a \in (0, \pi)$, as $n\to \infty$,
		\[
		\sup_{\beta \in K} e^{-e^\beta \mathrm{h}_n} \int_a^\pi \vert \cL_n(\beta +iu) \vert du = O \left(\frac{1}{n^\vp}\right),
		\]
		and is in particular a $o(1/\mathrm{h}_n^r)$ for all $r \in \N$.
	\end{enumerate}
\end{theorem}
\begin{proof}
	The first point is obvious since $\cL_n(\beta)$ is a finite sum. The convergence in the second point and the upper bound of the third point stem from Proposition \ref{prop cv unif martingale} and from Lemma \ref{lemme equivalent C n}. Finally the fourth point is just Lemma \ref{lemme controle queue integrale}. The fact that $1/n^\vp = o(1/\mathrm{h}_n^r)$ for all $r \in \N$ comes from Lemma \ref{lemme equivalent ps h n} which shows that $\mathrm{h}_n \sim \log n$ almost surely.
\end{proof}
Thus, the weighted profile $(\widetilde{\mathbb{L}}_n(k))_{k\ge 0}$ satisfies the conclusion of Theorem 2.1 of \cite{KMS17} with $w_n = \mathrm{h}_n$ and $\varphi(z)=e^z$. As in Equation (13) of \cite{KMS17}, let us write for all $n\in \N$ and $\beta \in \R$,
\[
x_n(k) =x_n(k;\beta)= \frac{k-e^\beta \mathrm{h}_n}{e^{\beta/2} \sqrt{\mathrm{h}_n}}.
\]
Then, Theorem 2.1 of \cite{KMS17} and Theorem \ref{th edgeworth expansion} show that for all $r \in \N$, for all $\beta_-< \beta_+ \le 1$, on the event that $M_\infty$ does not vanish on $(\beta_-, \beta_+)$, for every compact set $K \subset (\beta_-, \beta_+)$, 
\begin{equation}\label{eq edgeworth expansion}
	\mathrm{h}_n^{(r+1)/2}  \sup_{k \in \Z} \sup_{\beta \in K} \left\vert e^{\beta k - e^\beta \mathrm{h}_n} \widetilde{\mathbb{L}}_n(k) - \frac{W_\infty(\beta) e^{-\frac{1}{2} x_n(k)^2}}{e^{\beta/2}\sqrt{2\pi \mathrm{h}_n}} \sum_{j=0}^r \frac{G_j(x_n(k); \beta)}{\mathrm{h}_n^{j/2}} \right\vert\mathop{\longrightarrow}\limits_{n\to \infty}0,
\end{equation}
where the $G_j$'s are polynomials of degree $3j$ which are defined in Equation (16) of \cite{KMS17}. We will only use that $G_0=1$.

Recall that $\widetilde{\mathbb{L}}_n(k)=\mathbb{L}_n(k)e^{-S_k}$. By specializing to the case $r=0$ and $\beta=0$, we get the following local limit result: almost surely, as $n\to \infty$,
\begin{equation}\label{eq limite locale profil}
\mathbb{L}_n(k) = \frac{W_\infty(0)e^{\mathrm{h}_n}}{\sqrt{2\pi \mathrm{h}_n}} \exp\left(S_k- \frac{1}{2} \left(\frac{k-\mathrm{h}_n}{\sqrt{\mathrm{h}_n}}\right)^2 \right)  + e^{S_k}o\left(\frac{e^{\mathrm{h}_n}}{\sqrt{\mathrm{h}_n}} \right),
\end{equation}
where the error term is uniform in $k\in \N$. Another simple consequence of \eqref{eq edgeworth expansion} is obtained by taking $r=0$ and $k = \lfloor e^\beta \mathrm{h}_n\rfloor$: for every compact set $K \subset (-\infty,1)$, on the event that $W_\infty(z) \neq 0$ for all $z \in K$, we have almost surely for all $z \in K$,
\begin{equation}\label{eq limite profil h n}
\mathbb{L}_n(\lfloor e^z\mathrm{h}_n\rfloor) = \exp\left({S_{\lfloor e^z\mathrm{h}_n\rfloor} +  (1-z)e^z \mathrm{h}_n -\frac{1}{2} \log \mathrm{h}_n + O(1) }\right),
\end{equation}
where the $O(1)$ is uniform in $z \in K$. Using these results, let us prove Theorem \ref{theoreme profil}.

\begin{proof}[Proof of Theorem \ref{theoreme profil}]
	Let $\mathcal{Z} \coloneqq \{ z \in (-\infty, 1); \ M_\infty(z) =0\} = \{ z \in (-\infty, 1); \ W_\infty(z) =0\} $. By Lemma \ref{lemme limite non nulle}, we know that a.s.\@ $M_\infty(0)\neq 0 $ so that in particular the set $\mathcal{Z}$ of the zeros of the analytic function $z \mapsto M_\infty(z)$ is at most countable and its points are isolated. The first two points then follow from \eqref{eq limite locale profil} and \eqref{eq limite profil h n}.

	Let us then show the last part of the theorem. The a.s.\@ convergence of $\mathrm{h}_n/\log n$ stems from Lemma \ref{lemme equivalent ps h n}. Let $z\in (0,2)$ and $C>e^z$. We have
	\[
	\sum_{ e^z \log n\le k \le C \log n} \mathbb{L}_n(k)  \le  \sum_{ e^z \log n\le k \le C \log n} e^{zk - ze^z \log n} \mathbb{L}_n(k).
	\]
	Let $\vp>0$. By \eqref{condition ergodicite}, there exists $C_\vp >0$ such that $C_\vp n^\vp e^{-S_{k}}\ge 1$ for all $k \in \lb0, \lfloor C \log n \rfloor \rb$. Thus, 
	\[
	\sum_{ e^z \log n\le k \le C \log n} \mathbb{L}_n(k)  \le C_\vp n^\vp \sum_{ e^z \log n\le k \le C \log n} e^{zk - ze^z \log n} \mathbb{L}_n(k) e^{-S_k}\le C_\vp n^{\vp -ze^z \log n} \cL_n(z).
	\]
	But, by Lemma \ref{lemme equivalent C n} and since $(M_n(z))_{n\ge n_0}$ is a non-negative martingale (which hence converges almost surely), almost surely as $n\to \infty$ we have
	\begin{equation}\label{eq equivalement L n avec exp h n}
	\cL_n(z)  = C_n(z) M_n(z)  =e^{e^z\mathrm{h}_n + O(1)} M_\infty(z).
	\end{equation}
	Thus, using Lemma \ref{lemme equivalent ps h n}, for all $z \in (0,2)$, a.s.\@ for all $n$ large enough,
	\begin{equation}\label{eq pas de sommet trop haut 0}
		\sum_{ e^z \log n\le k \le C \log n} \mathbb{L}_n(k)  \le n^{2 \vp+ (1-z)e^z}
	\end{equation}
	By taking $\vp$ small enough and $z>1$, we see that the right-hand side goes to zero as $n\to \infty$. 
	We deduce that almost surely, for all $n$ large enough, there is no vertex of $\T_n$ at depth between $e^z \log n$ and $C\log n$. Furthermore, notice that if there is a vertex at depth larger than $C(\log(n)-1)$, then there is a vertex at depth $k$ for all $k\le C(\log(n)-1)$. As a result, 
	we conclude that a.s.\@ for all $n$ large enough, there is no vertex of $\T_n$ at depth at least $e^z \log n$.
	
	For the lower bound, one can see by \eqref{eq limite profil h n} that almost surely, for all $z<1$ which is not zero of $z \mapsto M_\infty(z)$, as $n\to \infty$,
	\[
	\mathbb{L}_n(\lfloor e^z \log n \rfloor) = \exp \left( (1-z)e^z \log n + o(\log n) \right),
	\]
	where we used Lemma \ref{lemme equivalent ps h n} and \eqref{condition ergodicite}. The right-hand side goes to infinity almost surely and is in particular non-zero. As a consequence, almost surely, $d(\T_n)/\log n \to e$. This is the last point of the theorem.
\end{proof}
Finally, let us conclude this section with the proof of Theorem \ref{th hauteur sommet typique}.
\begin{proof}[Proof of Theorem \ref{th hauteur sommet typique}]
	Let $V_n$ be a uniform random vertex of $\T_n$. Let $z \in (0,1)$. Note that $(1-z)e^z<1$. Therefore, by taking $\vp>0$ small enough, \eqref{eq pas de sommet trop haut 0} gives us that a.s.\@ as $n \to \infty$, for all $C>e^z$,
	\[
	\sum_{e^z \log n \le k \le C \log n} \mathbb{L}_n(k) = o(n).
	\]
	Thus, combining with the fact that, by Theorem \ref{theoreme profil}, $d(\T_n) \sim e \log n$ almost surely (so that for $C>e$, a.s.\@ for all $n$ large enough, there is no vertex at depth larger than $C\log n$), we deduce that the number of vertices at depth at least $e^z \log n$ is a $o(n)$ almost surely. Moreover, using that $\T_n$ has $n+1$ vertices, we see that the probability that $d(V_n)\ge e^z \log n$ converges to zero.
	
	Next, let us show the converse bound. We reason as in the proof of Theorem \ref{theoreme profil}. Let $z<0$. Note that for all $k \in \lb0, e^z \log n\rb$, we have $zk \ge ze^z \log n$, so that
	\[
	\sum_{0 \le k \le e^z \log n} \mathbb{L}_n (k)\le  \sum_{0 \le k \le e^z \log n}e^{zk - z e^z \log n} \mathbb{L}_n(k).
	\]
	Then, by \eqref{condition ergodicite}, there exists $C_\vp>0$ such that $C_\vp n^\vp e^{-S_k}\ge 1$ for all $k\in \lb0,\lfloor  e^z \log n \rfloor \rb$. So, the number of vertices with depth at most $e^z \log n$ is bounded from above by
	\[
	C_\vp n^{\vp -z e^z} \cL_n(z)  = C_\vp n^{\vp -z e^z} C_n(z)M_n(z) .
	\]
	But, by applying Lemma \ref{lemme equivalent C n}, the fact that the non-negative martingale $(M_n(z))_{n\ge n_0}$ converges a.s.\@ and Lemma \ref{lemme equivalent ps h n}, we see that a.s.\@ for all $n$ large enough,
	\[
	\sum_{0 \le k \le e^z \log n} \mathbb{L}_n (k) \le n^{2 \vp + (1-z)e^z}.
	\]
	By taking $\vp$ small enough and using that $(1-z)e^z<1$ since $z<0$, the right-hand side of the above display is a $o(n)$. In particular, we conclude as in the first part of the proof that the probability that $d(V_n) \le e^z \log n$ converges to zero. This ends the proof.
\end{proof}
\begin{remark}\label{more precise behaviour of h n}
	Let us say a few words on a refinement of the almost sure equivalent $\mathrm{h}_n \sim \log n$ as $n\to \infty$. Using the same ideas as in the above proof, one can see that for all $\vp>0$, almost surely, $\cL_n(0) \sim \sum_{(1-\vp)\log n \le k \le (1+\vp) \log n} \mathbb{L}_n(k)e^{-S_k}$ and $n\sim \sum_{(1-\vp)\log n \le k \le (1+\vp) \log n} \mathbb{L}_n(k)$ as $n\to \infty$. In particular, by the Borel Cantelli lemma, there exists a deterministic sequence $\vp_n \downarrow 0$ such that almost surely, as $n\to \infty$, we have 
	\[\cL_n(0) \sim \sum_{(1-\vp_n)\log n \le k \le (1+\vp_n) \log n} \mathbb{L}_n(k)e^{-S_k} \quad \text{and}\quad n\sim\sum_{(1-\vp_n)\log n \le k \le (1+\vp_n) \log n} \mathbb{L}_n(k).\] 
	Therefore, if we make the further assumption that $\max_{-\vp_n \log n \le k \le 1+\vp_n \log n} \vert S_{\lfloor \log n \rfloor +k}- S_{\lfloor \log n \rfloor}\vert = o(\vert S_{\lfloor \log n \rfloor}\vert)$ (which holds for i.i.d.\@ random weights), by taking \eqref{eq equivalement L n avec exp h n} into account with $z=0$, we deduce that a.s.\@ $\mathrm{h}_n=\log n - S_{\lfloor \log n \rfloor}+ O(1)+ o(\vert S_{\lfloor \log n\rfloor} \vert)$ as $n\to \infty$. 
\end{remark}

\section{A counterexample in the general case}\label{section contre-exemple}
This section is devoted to the proof of Proposition \ref{lemme contre-exemple} which gives a counterexample where the depth of the tree does not satisfy a scaling limit in the general case of weights bounded from above and from below which do not satisfy \eqref{condition ergodicite}. Here is the idea of the proof. Let $z\in \R$ with $z<2$. Recall that, by Lemma \ref{lemme martingale}, the process $(M_n(z))_{n\ge 0}$ is a positive martingale, so it converges almost surely to some random variable $M_\infty(z)$. In particular, for all $z <2$, for all $x \in [2, e]$, almost surely, 
\begin{equation*}
	\frac{1}{e^{S_{\lfloor x \log n \rfloor}}} \mathbb{L}_n(\lfloor x \log n \rfloor) e^{z \lfloor x \log n \rfloor} \le (1+o(1)) C_n(z) M_\infty(z) = O \left( e^{e^z \mathrm{h}_n } \right).
\end{equation*}
We will build a weight function $f$ such that the above inequality entails that $\mathbb{L}_n(\lfloor x \log n \rfloor)=0$ for $x\ge 5/2$ along a subsequence, thus proving the first part of Proposition \ref{lemme contre-exemple}. The second part will be obtained by choosing $f$ constant during long intervals and applying Theorem \ref{theoreme profil}.

Let us start with some observations. If $f$ is a weight function, we denote by $\T_n(f)$ the depth-weighted tree constructed using the weight function $f$. Let $A\in \N$ be a constant. Let $f_\infty$ be a weight function taking its values in $\{1,3\}$ such that $f_\infty(k)=3$ for all $k\ge A$. For all $n\in \N$ such that $2\lfloor \log n \rfloor \ge A$, let $f_n$ be a weight function taking its values in $\{1,3\}$ such that $f_n$ coincides with $f_\infty$ on $\lb 0, A \rb$, such that  $f_n(k)=3$ for all $k \in \lb A, 2\lfloor  \log n \rfloor\rb$ and $f_n(k)=1$ for all $k>2 \lfloor \log n\rfloor$. Recall that the root is denoted by $v_0$ and that for all $m\ge 1$, we denote by $v_m$ the vertex added at the $m$-th step.

\begin{lemma}\label{lemma coupling}
	Let $n \in \N$ be such that $2 \lfloor \log n \rfloor \ge A$. We can couple $(\T_m(f_\infty))_{m\ge 0}$ with $(\T_m(f_n))_{m\ge 0}$ so that  for all $m\in \lb 0, n \rb$, if the vertex $v_m$ is at depth at most $2 \lfloor \log n \rfloor $ in $\T_m(f_\infty)$, then $v_m$ has the same depth in both trees $\T_m(f_\infty)$ and $\T_m(f_n)$.
\end{lemma}

\begin{proof}
	We construct the coupling of $(\T_m(f_\infty))_{m\ge 0}$ with $(\T_m(f_n))_{m\ge 0}$ as follows. For all $m\ge 0$, we will denote by $d_n(v_m)$ and $d_\infty(v_m)$ the depth of the vertex $v_m$ in $\T_m(f_n)$ and $\T_m(f_\infty)$ respectively.
	
	We start with $\T_0(f_n)=\T_0(f_\infty)$, the tree with one root vertex $v_0$. Assume that we constructed $(\T_j(f_\infty))_{ 0\le j \le m}$ with $(\T_j(f_n))_{ 0\le j \le m}$ until some time $m\ge 0$ so that for all $j \in \lb 0, m \rb$, 
	if $d_\infty(v_j)\le 2 \lfloor \log n \rfloor$, then $d_n(v_j) = d_\infty(v_j)$. 
	Recall that we attach $v_{m+1}$ to a vertex $v$ in $\T_m(f_n)$ with probability
	\[
	\frac{f_n(d_n(v))}{\sum_{j=0}^{m} f_n(d_n(v_j))}.
	\]
	Note that since $f_n(k)=f_\infty(k)$ for all $k \le 2 \lfloor \log n \rfloor$ and since $f_n \le f_\infty$, we deduce that for all $v \in \T_m(f_n)$ such that $d_\infty(v) \le2 \lfloor \log n \rfloor$,
	\[
	\frac{f_n(d_n(v))}{\sum_{j=0}^{m} f_n(d_n(v_j))} \ge \frac{f_n(d_n(v))}{\sum_{j=0}^{m} f_\infty(d_n(v_j))} = \frac{f_\infty(d_n(v))}{\sum_{j=0}^{m} f_\infty(d_n(v_j))} = \frac{f_\infty(d_\infty(v))}{\sum_{j=0}^{m} f_\infty(d_n(v_j))} \ge \frac{f_\infty(d_\infty(v))}{\sum_{j=0}^{m} f_\infty(d_\infty(v_j))},
	\]
	where in the last equality, we used the fact that $d_n(v) = d_\infty(v)$ and in the last inequality, we used the fact that if $d_\infty(v_j) \le 2 \lfloor \log n \rfloor$, then $d_n(v_j)=d_\infty(v_j)$ and that otherwise we have $f_\infty(d_\infty(v_j))=3 \ge f_\infty(d_n(v_j))$.
	
	Thus, we can couple $(\T_j(f_\infty))_{ 0\le j \le m+1}$ with $(\T_j(f_n))_{ 0\le j \le m+1}$ so that if we attach $v_{m+1}$ to a vertex at depth at most $2 \lfloor \log n \rfloor$ in $\T_m(f_\infty)$, then we attach it to the same vertex in $\T_m(f_n)$. This proves the desired result.
\end{proof}

From the above lemma, we will deduce that the depth of a typical vertex in $\T_n(f_n)$ is roughly $\log n$.
\begin{lemma}\label{lemme hauteur sommet typique}
	For all $\vp>0$, under the same coupling as in Lemma \ref{lemma coupling}, almost surely, for all $n$ large enough,
	\[
	\forall m \in \lb \lfloor \log \log n \rfloor , n \rb , \ \#\{ v \in \T_m(f_n);  \ (1-\vp) \log m \le  d_n(v) \le (1+\vp) \log m \} \ge (1-\vp) m  .
	\]
\end{lemma}
\begin{proof}
	Recall that $\T_m(f_n)$ and $\T_m(f_\infty)$ have both $m+1$ vertices. Thanks to the coupling of Lemma \ref{lemma coupling}, it suffices to prove that for all $\vp>0$, a.s.\@ for all $n$ large enough,
	\[
	\forall m \in \lb \lfloor \log \log n \rfloor , n \rb , \ \#\{ v \in \T_m(f_\infty);  \ (1-\vp) \log m \le  d_\infty(v) \le (1+\vp) \log m \} \ge (1-\vp) m.
	\]
	Since the weight function $f\coloneqq f_\infty/3$ satisfies the condition \eqref{condition ergodicite}, we can apply Theorem \ref{th hauteur sommet typique} which concludes the proof.
\end{proof}
Another lemma which will be useful in the proof of Proposition \ref{lemme contre-exemple} is an upper bound for the depth of $\T_n(f_n)$.
\begin{lemma}\label{lemme hauteur au plus cinq demi}
	We have
	\[
	\P\left( d(\T_n(f_n)) \ge \frac{5}{2} \log n 
	 \right)
	 \mathop{\longrightarrow}\limits_{n\to \infty} 0.
	\]
\end{lemma}
\begin{proof}
	Recall that $n_0= \lfloor e^2/c \rfloor +1$ (since $f_n$ takes its values in $\{1,3\}$, we can take $c=1/3$). Let $S^n_k \coloneqq \sum_{j=0}^{k-1} \log f_n(j)$ for $k\ge 0$ be the walk associated with $f_n$ and let $(\cL^n_m(z))_{m\ge 0}, (C^n_m(z))_{m\ge n_0}$ and $(M^n_m(z))_{m\ge n_0}$ be the associated weighted Laplace transform, normalization and martingale for $z \in (0,2)$. Let $(\mathbb{L}^n_m(k))_{k\in \N}$ be the profile of $\T_m(f_n)$. We also denote by $Z^n_m\coloneqq \sum_{v \in \T_m(f_n)} f_n(d_n(v))$ the sum of the weights for $\T_m(f_n)$ and we set
	\begin{equation}\label{eq def h n m}
	\mathrm{h}^n_m \coloneqq \sum_{k=n_0}^{m-1} \frac{1}{Z^n_k}.
	\end{equation}
	Note that $3(m+1)\ge Z^n_m \ge m+1$ for all $n$ such that $2 \lfloor \log n \rfloor \ge A$ and for all $m \ge 0$. As a result, by exactly the same proof as in the proof of Lemma \ref{lemme equivalent C n}, there is a constant $b>0$ such that for all $z \in (0,2)$, for all $n\ge n_0$ such that $2 \lfloor \log n \rfloor\ge A$, for all $m\ge n_0$,
	\[
	\left\vert \log C^n_m(z) -e^z \mathrm{h}^n_m \right\vert <b.
	\]
	So, for all $x \ge 0$, for all $z \in (0,2)$, for all $n\ge n_0$ such that $2 \lfloor \log n \rfloor\ge A$, using that $\cL^n_n(z)= C^n_n(z) M^n_n(z)$, we have
	\[
	\mathbb{L}^n_n(\lfloor x \log n \rfloor) \le e^{S^n_{\lfloor x \log n \rfloor} - z \lfloor x \log n \rfloor} \cL^n_n(z) 
	 \le e^{S^n_{\lfloor x \log n \rfloor} - z \lfloor x \log n \rfloor + e^z \mathrm{h}^n_n + b} M^n_n(z).
	\]
	Let $\vp>0$. Then, we have
	\[
	\mathbb{L}^n_n(\lfloor x \log n \rfloor) {\bf 1}_{\mathrm{h}_n^n \le (\vp+1/3) \log n}
	\le e^{S^n_{\lfloor x \log n \rfloor} - z \lfloor x \log n \rfloor + e^z \left(\frac{1}{3}+ \vp \right) \log n+ b} M^n_n(z) {\bf 1}_{\mathrm{h}_n^n \le (\vp+1/3) \log n}.
	\]
	By taking the expectation and using the fact that $\E(M^n_{n_0}(z)) = \E(\cL^n_{n_0}(z)) \le (n_0+1)e^{n_0 z + n_0 \log 3}$, we see that as $n\to \infty$,
	\[
	\E\left( \mathbb{L}^n_n(\lfloor x \log n \rfloor)  {\bf 1}_{\mathrm{h}_n^n \le (\vp+1/3) \log n}\right) =O \left(
	\exp\left({S^n_{\lfloor x \log n \rfloor} - z \lfloor x \log n \rfloor + e^z \left(\frac{1}{3}+ \vp \right) \log n }\right)  \right).
	\]
	We take $x=5/2$ and $z=3/2$. Then, using that for all $k \in \lb A, 2 \lfloor \log n \rfloor\rb$, we have $f_n(k)=3$ and that $f_n(k) = 1$ for all $k>2 \lfloor \log n \rfloor$,
	\begin{align*}
		S^n_{\lfloor x \log n \rfloor} - z \lfloor x \log n \rfloor + e^z \left(\frac{1}{3} + \vp \right) \log n &= \left( 2\log (3)- \frac{5}{2} z + e^z  \left(\frac{1}{3} + \vp \right) \right)  \log n  + o(\log n)\\
		&=\left( 2\log (3)- \frac{15}{4}  + e^{3/2}  \left(\frac{1}{3} + \vp \right) \right)  \log n  + o(\log n),
	\end{align*}
	where we used the equivalent $S^n_{\lfloor x \log n \rfloor} \sim 2\log(3)\log n$, which comes from the fact that $f_n(k)=1$ for all $k>2 \lfloor \log n \rfloor$ and $f_n(k)=3$ for all $k \in \lb A, 2 \lfloor \log n \rfloor\rb$. 
	
	Using that $2\log(3) -15/4 + e^{3/2}/3 <0$, by taking $\vp>0$ small enough, we obtain that
	\begin{equation}\label{eq esperance profil tend vers zero}
	\E\left( \mathbb{L}^n_n\left(\left\lfloor \frac{5}{2} \log n \right\rfloor\right)  {\bf 1}_{\mathrm{h}_n^n \le (\vp+1/3) \log n}\right)  \mathop{\longrightarrow}\limits_{n \to \infty} 0.
	\end{equation}
	Furthermore, by applying Lemma \ref{lemme hauteur sommet typique} and the fact that for all $k \in \lb A, 2 \lfloor \log n \rfloor\rb$, we have $f_n(k)=3$, we see that under the same coupling as in Lemma \ref{lemma coupling},
	\[
	\sup_{\log \log n \le m \le n} \left\vert \frac{Z^n_m}{3m} -1 \right\vert \mathop{\longrightarrow}\limits_{n \to \infty}^{\mathrm{a.s.}} 0.
	\]
	Recalling \eqref{eq def h n m}, we deduce that
	\[
	\P\left(\mathrm{h}_n^n \le \left( \frac{1}{3} +\vp\right)\log n \right) \mathop{\longrightarrow}\limits_{n\to \infty} 1.
	\]
	By taking \eqref{eq esperance profil tend vers zero} into account, this concludes the proof of the lemma.
\end{proof}
Using the above lemmas, we are now in position to provide the counterexample mentioned in Proposition \ref{lemme contre-exemple}.
\begin{proof}[Proof of Proposition \ref{lemme contre-exemple}]
	We define $f$ as follows. The idea is to alternate between a long sequence of $1$'s so that the limit superior of $d(\T_n)/\log n$ is at least $e$ thanks to Theorem \ref{theoreme profil} and a ``copy of $f_n$'' so that the limit inferior is smaller than $5/2$ thanks to Lemma \ref{lemme hauteur au plus cinq demi}. Let $(n_j)_{j\ge 1}$ be an increasing sequence of integers constructed as follows. We choose $n_1$ large enough and set $f(k)=1$ for all $k \in \lb 0, n_1 \rb$ so that the probability that $d(\T_{n_1})/\log n_1 \ge e-1/2$ is at least $1/2$ thanks to Theorem \ref{theoreme profil}. Then, let $A_1=n_1+1$. Choose $n_2$ large enough that if we set $f(k)=3$ for all $k \in \lb A_1, 2 \lfloor \log n_2 \rfloor \rb$ and $f(k)=1$ for all $k \in \lb 2 \lfloor \log n_2 \rfloor +1, n_2\rb$, then by Lemma \ref{lemme hauteur au plus cinq demi}, the probability that $d(\T_{n_2}) \le (5/2)\log n_2$ is at least $1/2$. Similarly, for all $j\ge 1$, assuming that $n_1, \ldots, n_{2j}$ have been constructed, we take $n_{2j+1}$ large enough that by setting $f(k)=1$ for all $k \in \lb n_{2j}, n_{2j+1} \rb$, we have $d(\T_{n_{2j+1}})/\log n_{2j+1} \ge e-1/2^{j}$ with probability at least $1-1/2^{j}$ thanks to Theorem \ref{theoreme profil}. Write $A_j=n_{2j+1}+1$. Then, we take $n_{2j+2}$ large enough so that, by setting $f(k)=3$ for all $k \in \lb A_j, 2 \lfloor \log n_{2j+2} \rfloor\rb$ and $f(k)=1$ for all $k \in  \lb 2 \lfloor \log n_{2j+2} \rfloor+1 , n_{2j+2} \rb$, we obtain that $d(\T_{n_{2j+2}})/\log n_{2j+2} \le 5/2$ with probability at least $1-1/2^j$. We conclude using the Borel-Cantelli lemma.
\end{proof}

\paragraph{Acknowledgements.} I acknowledge the support of a Research Fellowship from Emmanuel College, Cambridge.

\bibliographystyle{alpha}
\bibliography{biblio}

\end{document}